\documentclass[12pt,leqno]{amsart}
\usepackage{pifont}
\usepackage{stmaryrd}
\usepackage{dsfont}
\usepackage{mathrsfs}
\usepackage{amsmath}
\usepackage{amssymb}
\usepackage{hyperref}
\usepackage{bookmark}
\usepackage[bottom]{footmisc}
\usepackage{verbatim}

\def\g{\mathfrak{g}}

\def\pa{\partial}

\def\Z{\mathbb{Z}}
\def\N{\mathbb{N}}

\def\C{\mathbb{C}}

\numberwithin{equation}{section}
\newtheorem{theo}{Theorem}[section]
\newtheorem{defi}[theo]{Definition}
\newtheorem{coro}[theo]{Corollary}
\newtheorem{lemm}[theo]{Lemma}
\newtheorem{prop}[theo]{Proposition}

\newtheorem{exa}[theo]{Example}

\newtheorem{remark}[theo]{Remark}
\allowdisplaybreaks

\begin{document}

\title[Cocommutative vertex bialgebras]{Cocommutative vertex bialgebras}

\author{Jianzhi Han, Haisheng Li, Yukun Xiao}

\address{School of Mathematical Sciences, Tongji University, Shanghai, 200092, China.}
\email{jzhan@tongji.edu.cn}

\address{Department of Mathematical Sciences, Rutgers University, Camden, NJ 08102, USA.}
\email{hli@camden.rutgers.edu}

\address{School of Mathematical Sciences, Tongji University, Shanghai, 200092, China.}
\email{ykxiao@tongji.edu.cn}

\subjclass[2010]{17B69, 17B68, 16T10}

\keywords{ Bialgebra, vertex bialgebra,  vertex Lie algebra.}

%\thanks{This work is supported by National Natural Science Foundation of China (Grant Nos. 11801363, 11771279 and 11671247).}

\begin{abstract}
In this paper, the structure of cocommutative vertex bialgebras is investigated. For a general vertex bialgebra  $V$, 
it is proved that the set $G(V)$ of group-like elements is naturally an abelian semigroup,
whereas the set $P(V)$ of primitive elements is a vertex Lie algebra. 
For $g\in G(V)$, denote by $V_g$ the connected component containing $g$.
Among the main results, it is proved  that if $V$ is a cocommutative vertex bialgebra, then $V=\oplus_{g\in G(V)}V_g$, where
$V_{\bf 1}$ is a vertex subbialgebra which is isomorphic to the vertex bialgebra ${\mathcal{V}}_{P(V)}$ 
associated to the vertex Lie algebra $P(V)$, and $V_g$ is a $V_{\bf 1}$-module for $g\in G(V)$.
 In particular, this shows that every cocommutative connected vertex bialgebra
$V$ is isomorphic to ${\mathcal{V}}_{P(V)}$ and hence establishes the equivalence between the category of
cocommutative connected vertex bialgebras  and the category of vertex Lie algebras.
Furthermore, under the condition that $G(V)$ is a group and lies in the center of $V$, 
it is proved that  $V={\mathcal{V}}_{P(V)}\otimes \C[G(V)]$
as a coalgebra where the vertex algebra structure is explicitly determined.
\end{abstract}

\maketitle

\section{Introduction}
Vertex algebras are analogous and closely related to classical algebraic systems such as Lie algebras 
and associative algebras on one hand,  and they are also highly nonclassical in nature on the other hand.
Especially, vertex algebras can be associated to (infinite-dimensional) Lie algebras of a certain type,
 including the Virasoro algebra and the affine Lie algebras.
Vertex algebra analogues of Lie algebras, called vertex Lie algebras, were studied by Primc (see \cite{P}),
and the same structures, called Lie conformal algebras, were  studied independently by Kac (see \cite{Kac}).
Associated to each vertex Lie algebra $C$, one has a Lie algebra ${\mathcal{L}}_C$. 
(Another notion of vertex Lie algebra was studied in \cite{DLM}, where
vertex Lie algebras are defined as Lie algebras of a certain type.)
Just as an associative algebra is naturally a Lie algebra, a vertex algebra is naturally a vertex Lie algebra.
On the other hand, for any vertex Lie algebra $C$,  one has a vertex algebra ${\mathcal{V}}_{C}$,
which is constructed on a generalized Verma module for ${\mathcal{L}}_C$ and contains $C$ as a vertex Lie subalgebra.
A result due to Primc is that for any vertex algebra $V$ and for any vertex Lie algebra homomorphism $\psi: C\rightarrow V$,
there exists a (unique) vertex algebra homomorphism $\Psi: {\mathcal{V}}_{C}\rightarrow V$ such that $\Psi|_C=\psi$. 
This illustrates the relationship between vertex algebras and Lie algebras.

It has been well understood that vertex algebras are analogues and generalizations of 
commutative associative algebras with identity.
Noncommutative associative algebra analogues of vertex algebras were studied by Baklov and Kac 
under the name ``field algebra'' (see \cite{BK}) and independently by the second named author of 
this current paper under the name ``$G_1$-axiomatic vertex algebra'' in \cite{L0}.
These were further studied under the name ``nonlocal vertex algebras'' in \cite{L1}. 

 Theoretically, vertex algebra analogues of coalgebras, bialgebras, and Hopf algebras are also important objects to study.
 In \cite{Hubb1, Hubb2}, Hubbard  introduced and studied a notion of vertex operator coalgebra, 
 which is dual to that of vertex operator algebra.
Then the second named author of the current paper initiated a study on vertex algebra analogues of bialgebras in \cite{L}, 
where a notion of (nonlocal) vertex bialgebra was introduced.
By definition, a (nonlocal) vertex bialgebra is simply a (nonlocal) vertex algebra equipped with a classical coalgebra structure
such that the comultiplication map and the counit map are homomorphisms of (nonlocal) vertex algebras.
(Note that the vertex coalgebra maps of a vertex operator coalgebra in the sense of Hubbard can not be interpreted 
as vertex algebra morphisms.)
 Furthermore, a notion of (nonlocal) vertex module-algebra for a (nonlocal) vertex bialgebra was introduced
 and a smash product construction of (nonlocal) vertex algebras was obtained therein. 
Some initial applications were also given. 

The study of vertex bialgebras was continued in recent work \cite{JKLT}. 
For a vertex bialgebra $H$,  a notion of right $H$-comodule nonlocal vertex algebra was introduced
and a general construction of quantum vertex algebras was obtained 
by using a right comodule nonlocal vertex algebra structure and 
a compatible (left) module vertex algebra structure on a nonlocal vertex algebra. 
As an application, a family of deformations of lattice vertex algebras was obtained. 
Currently, this is used in a sequel to establish a natural connection between twisted quantum affine algebras 
and quantum vertex algebras.

We mention that Hopf action on vertex operator algebras 
was studied by Dong and Wang in \cite{DW}, and studied further by Wang (see \cite{W1,W2}),
 from a different direction with different perspectives. 

This current work is part of a program to fully investigate the structure of vertex bialgebras.
Among the main results of this paper,  we prove that if $V$ is a cocommutative vertex bialgebra, 
then $V=\oplus_{g\in G(V)}V_g$, where $V_{\bf 1}$ is a vertex subbialgebra 
which is isomorphic to the vertex bialgebra ${\mathcal{V}}_{P(V)}$ 
associated to the vertex Lie algebra $P(V)$, and $V_g$ is a $V_{\bf 1}$-module for $g\in G(V)$. 
In particular, this shows that every cocommutative connected vertex bialgebra
$V$ is isomorphic to ${\mathcal{V}}_{P(V)}$ and establishes the equivalence between the category of
cocommutative connected vertex bialgebras  and the category of vertex Lie algebras.
Furthermore, under the condition that $G(V)$ is a group and lies in the center of $V$, 
we prove that  $V={\mathcal{V}}_{P(V)}\otimes \C[G(V)]$ (canonically) 
as a coalgebra, where the vertex algebra structure is explicitly determined.

Now, we describe the main contents of this paper with more technical details. 
Let us start with some well known facts in the classical Hopf algebra theory.
First of all,  as we work on $\C$, any cocommutative coalgebra is pointed, so that an irreducible coalgebra 
is the same as a connected coalgebra. Among the important examples of Hopf algebras are the group algebras
(of groups) and the universal enveloping algebras of  Lie algebras.
The closely related notions of group-like element and primitive element play a crucial role.
A basic fact is that for every bialgebra $B$, the set $G(B)$ of group-like elements
is a semigroup with identity, whereas the set $P(B)$ of primitive elements is a Lie algebra.
Another fact is that a cocommutative bialgebra $B$ is a Hopf algebra if and only if $G(B)$ is a group.
A classical result is that  every cocommutative (pointed)  Hopf algebra $H$ is canonically isomorphic to 
the smash product $U(P(H))\sharp \mathbb{C}[G(H)]$ with respect to a natural action of $G(H)$ on $U(P(H))$ (see \cite{M}). 
This in particular shows that any cocommutative connected  bialgebra $H$ is isomorphic to  
the universal enveloping algebra $U(P(H))$.
In terms of categories, this gives the equivalence between the category of cocommutative connected bialgebras 
and that of Lie  algebras.

In this paper, to begin we show that for every vertex bialgebra $V$, $P(V)$ is a vertex Lie algebra and $G(V)$ is 
 an abelian semigroup with identity. While the proof for the part on $P(V)$ is straightforward, 
 the proof for the part on $G(V)$ is nonclassical. 
Note that for a classical bialgebra, the left multiplication associated to a group-like element 
is a coalgebra morphism, due to the fact that the multiplication map is a coalgebra morphism.
However, this is not the case for a vertex bialgebra; the ``vertex multiplication'' is {\em not} a coalgebra morphism.
 As the first key step we prove (Proposition \ref{group-like-group}) that 
 if $g,h\in G(V)$, then $g_nh=0$ for all $n\ge 0$, or equivalently $[Y(g,x_1),Y(h,x_2)]=0.$
Then we show that $G(V)$ is an abelian semigroup with the operation defined by $gh=g_{-1}h$ for $g,h\in G(V)$.
We also show that the vertex subalgebra of $V$ generated by $G(V)$ is a commutative and cocommutative vertex subbialgebra.

For $g\in G(V)$, let $V_g$ denote the connected component of $V$ containing $g$. 
Assume that $V$ is a cocommutative vertex bialgebra. By a classical result we have $V=\oplus_{g\in G(V)}V_g$.
As the first main result (Theorem \ref{main-result1}), we prove that $V_{\bf 1}$ is a vertex subbialgebra 
while $V_g$ is a module for $V_{\bf 1}$ as a vertex algebra for every $g\in G(V)$.
Note that the proof in the classical case no longer works for the same reason that
the ``vertex multiplication'' is {\em not} a coalgebra morphism. 
The main idea here is to employ the universal enveloping algebra of the Lie algebra ${\mathcal{L}}_{P(V)}$.
We prove that the natural vertex algebra morphism $\Psi: {\mathcal{V}}_{P(V)}\rightarrow V$ 
with $\Psi|_{P(V)}=1$ gives a vertex bialgebra isomorphism from ${\mathcal{V}}_{P(V)}$ to $V_{\bf 1}$.
This in particular implies that any cocommutative connected vertex bialgebra $V$ is isomorphic to
${\mathcal{V}}_{P(V)}$ (Corollary \ref{vla-identification}). 
Furthermore, assuming that $G(V)$ is a group and lies in the center of $V$,
we prove that $V={\mathcal{V}}_{P(V)}\otimes \C[G(V)]$ as a coalgebra, on which 
 the corresponding vertex algebra structure is explicitly exhibited (Theorem \ref{thm-second}). 
Motivated by this result,  we give a general construction of vertex (bi)algebras (Theorem \ref{mixed-vla-group});
we construct a vertex (bi)algebra $V\otimes_{\phi}\C[L]$
from a general vertex (bi)algebra $V$ and an abelian semigroup $L$ with identity 
with a map $\phi: L\rightarrow P(V)$ satisfying certain conditions.

Note that the category of commutative vertex algebras is canonically isomorphic to that of
 commutative differential algebras (which are commutative associative algebras equipped with one derivation).
 Commutative and cocommutative vertex bialgebras are essentially commutative and cocommutative differential bialgebras.
 This naturally leads us to a study on commutative and cocommutative differential bialgebras.
 To any abelian semigroup $L$ with identity, we associate a commutative and cocommutative differential bialgebra $B_L$
 and we obtain a characterization of $B_L$.

This paper is organized as follows: In Section 2, we recall basic notions and
 some fundamental results in the classical Hopf algebra theory, including 
 the basic results on group-like elements, primitive elements,  irreducible coalgebras and connected coalgebras.
In Section 3, we prove that for any vertex Lie  algebra $C$,
the associated vertex algebra ${\mathcal{V}}_C$ is a cocommutative connected vertex bialgebra with $P({\mathcal{V}}_C)=C$. 
Section 4 is the core of this paper, in which we prove that for any vertex bialgebra $V$, 
$G(V)$ is an abelian  semigroup with identity and  especially we prove that  every cocommutative connected vertex bialgebra $B$ 
is canonically isomorphic to the associated vertex bialgebra  ${\mathcal{V}}_{P(B)}$.
In Section 5, we study commutative and cocommutative vertex bialgebras, or equivalently
commutative and cocommutative differential bialgebras. 
In particular, we study those associated to abelian semigroups, in order 
to determine the structure of commutative and cocommutative vertex bialgebras.

For this paper, we work on the field $\C$ of complex numbers;  all vector spaces are assumed to be over $\mathbb{C}$.
In addition to the standard notation $\mathbb Z$ for the set of integers, we use $\mathbb N$ for the set of nonnegative integers
and $\Z_{+}$ for the set of positive integers.

\section{Preliminaries}

This section is preliminary.
In this section, we recall the basic notions about classical 
coalgebras and bialgebras, and collect the basic results we shall need.
References \cite{Abe}, \cite{M}, and \cite{Sw}  have been good resources for us,
and we refer the readers to these books for details.

A \emph{coalgebra} is a vector space $M$ equipped with linear maps $\Delta: M \rightarrow M \otimes M$, 
called the \emph{comultiplication}, and $\varepsilon: M \rightarrow \C$, called the \emph{counit}, such that
$$
\begin{aligned}
({\rm id}_M \otimes \Delta) \Delta &=(\Delta \otimes {\rm id}_M) \Delta, \\
\sum \varepsilon(a_{(1)}) a_{(2)} &=a=\sum \varepsilon(a_{(2)}) a_{(1)}\quad {\rm for}\ a\in M,
\end{aligned}
$$
 where $\Delta(a) = \sum a_{(1)}\otimes a_{(2)}$ in the Sweedler notation.
The comultiplication and the counit are often referred to as the {\em coalgebra structure maps}.

 A coalgebra $M$ is said to be \emph{cocommutative} if
 \begin{align*}
 \sum a_{(1)}\otimes a_{(2)}=\sum a_{(2)}\otimes a_{(1)}\quad  \text{for every }a\in M.
 \end{align*}

 A {\it subcoalgebra} of a coalgebra $M$ is a subspace $N$ such that
 $N$ with respect to the coalgebra structure of $M$ becomes a coalgebra itself,
 or equivalently, $\Delta(N) \subset N \otimes N$.
 A nonzero coalgebra $M$ is said to be \emph{simple} if $M$ and $\{0\}$ are the only subcoalgebras.
 
 \begin{remark}
 {\em The Fundamental Theorem of Coalgebras states that every coalgebra is the sum of its finite-dimensional subcoalgebras
 (cf. \cite{Abe}). %[Corollary 2.2.14]
This implies that every simple coalgebra is finite-dimensional. Furthermore, 
every nonzero minimal subcoalgebra of a coalgebra is simple and finite-dimensional.}
\end{remark}

 \begin{defi}\label{def-irreducible}
 {\em A coalgebra is said to be {\em irreducible} if it has a unique simple subcoalgebra. 
Furthermore, a maximal irreducible subcoalgebra is called an \emph{irreducible component}.}
\end{defi}

The following results can be found in \cite{Abe}: %[Theorem 2.4.7; page 97]

\begin{theo}\label{Abe-thm2.4.7}
Let $C$ be a coalgebra. Then

(1) Every irreducible subcoalgebra of $C$ is contained in an irreducible component of $C$.

(2) A sum of distinct irreducible components is a direct sum.

(3) If $C$ is cocommutative, then $C$ is the direct sum of its irreducible components. 
\end{theo}

Note that a simple coalgebra is an irreducible coalgebra from definition.
It is clear that every one-dimensional subcoalgebra of a coalgebra is simple.

\begin{defi}\label{def-pointed}
{\em  A coalgebra $M$ is said to be \emph{pointed} if every simple subcoalgebra of $M$ is one-dimensional.}
\end{defi}

\begin{remark}
{\em Note that a coalgebra version of the Nullstellensatz states that every cocommutative coalgebra over an
algebraically closed field is pointed.}
\end{remark}

\begin{defi}\label{conncted def}
\rm	For a coalgebra $M$, denote by $M^{(0)}$  the sum of all simple subcoalgebras of $M$,
which is called the \emph{coradical}  of $M$.
A coalgebra $M$ is said to be \emph{connected} if $M^{(0)}$ is one-dimensional, or equivalently, if $M$ is irreducible and pointed.	
\end{defi}

Let $M$ be a coalgebra.  An element $g$ of $M$ is said to be {\em group-like} if $\Delta(g)=g\otimes g$ and $\varepsilon(g)=1$.
Let $G(M)$ be the set of all group-like elements of $M$:
\begin{align}
G(M)=\{ g\in M\ |\   \Delta(g)=g\otimes g,\  \varepsilon(g)=1\}.
\end{align}
Note that each group-like element spans a one-dimensional (simple) subcoalgebra and 
on the other hand, every one-dimensional subcoalgebra is spanned by
a group-like element. Thus, a coalgebra $M$ is pointed if and only if every simple subcoalgebra is spanned by
a group-like element.  Furthermore, every nonzero pointed coalgebra $M$ contains at least one group-like element,
and every connected coalgebra $C$ has one and only one group-like element.

For $g\in G(M)$ with $M$ a nonzero coalgebra, let $M_g$ denote the irreducible (connected) component containing $g$.
It is known that $\oplus_{g\in G(M)}M_g$ is a subcoalgebra of $M$ and furthermore,
if $M$ is also cocommutative, then $M=\oplus_{g\in G(M)}M_g$.

 Let $(M,\Delta_M,\varepsilon_M)$ and $(N,\Delta_N,\varepsilon_N)$ be two coalgebras.
   A \emph{coalgebra morphism} from $M$ to $N$
  is a linear map $f: M\rightarrow N$ such that
  \begin{align}
  \Delta_N  f=(f \otimes f) \Delta_M\  \text{ and }\  \varepsilon_M=\varepsilon_N  f.
  \end{align}

The following are some basic results on connected coalgebras
 (cf. \cite{Abe}): % [Corollary 2.4.21, Theorem 2.4.22, page 103]:  

\begin{theo}\label{Abe-C2.4.21-T2.4.22}
 (i) The image of a connected  coalgebra under a coalgebra morphism is connected. 
  (ii) If $C$ and $E$ are connected coalgebras, then $C\otimes E$ is connected.
 %Let $C$ and $E$ be irreducible coalgebras with the unique simple subcoalgebras $R$ and $S$, respectively. 
%Then (iii) $C\otimes E$ is irreducible if and only if $R\otimes S$ is irreducible.
%(iv) If $C$ is pointed irreducible (connected), then $R\otimes S$ is a simple subcoalgebra of $C\otimes E$  
%and $C\otimes E$ is pointed irreducible.
\end{theo}

A \emph{bialgebra} is an associative  algebra with an identity, equipped with a coalgebra structure such that
the two coalgebra structure maps are associative algebra morphisms.
 A \emph{bialgebra morphism} is a linear map 
which is both an algebra morphism and a coalgebra morphism.

A \emph{Hopf algebra} is a bialgebra $H$ equipped with a linear endomorphism $S: H\rightarrow H$, called the \emph{antipode}, 
such that 
\begin{align}
\sum S(c_{(1)}) c_{(2)}=\varepsilon(c) 1=\sum c_{(1)} S(c_{(2)})\quad \text{ for  }c\in H.
\end{align}

\begin{exa}\label{UEA-la}
{\em Let $\g$ be any Lie algebra. Then the universal enveloping algebra $U(\g)$ is a cocommutative Hopf algebra,
where the Hopf algebra structure is uniquely determined by
\begin{align*}
\Delta(x)=x\otimes 1+1\otimes x,\quad \varepsilon(x)=0, \quad S(x)=-x \quad\text{  for  }x\in \g.
\end{align*}
Let $x\in \g$ be a nonzero vector. Set $x^{(0)}=1$ and $x^{(n)}=\frac{1}{n!}x^n\in U(\g)$ for $n\ge 1$. Then
\begin{align}\label{x-divided-powers}
\Delta(x^{(n)})=\sum_{i=0}^{n}x^{(n-i)}\otimes  x^{(i)},\quad \varepsilon(x^{(n)})= \delta_{n,0}
\quad \text{ for }n\in \N.
\end{align}}
\end{exa}

\begin{exa}\label{group-algebra}
{\em Let $G$ be any semigroup with identity $1$. 
Then the semigroup algebra $\C[G]$ is a bialgebra where the coalgebra structure is given by
\begin{align}
\Delta(g)=g\otimes g,\  \  \varepsilon(g)=1\quad \text{ for }g\in G.
\end{align}
Furthermore, if $G$ is a group, then $\C[G]$ is a Hopf algebra where the antipode $S$ is given by
\begin{align}
S(g)=g^{-1}\quad \text{ for }g\in G.
\end{align}}
\end{exa}

The following are some basic facts about group-like elements (cf. \cite{Abe}): %[Theorem 2.1.2]
 
\begin{prop}
(1) For any coalgebra $M$, $G(M)$ is a linearly independent subset. 
(2) For any bialgebra $B$, $G(B)$ is a semigroup with an identity and $\C[G(B)]$ is a subbialgebra.
(3)  A cocommutative bialgebra $B$ is a Hopf algebra if and only if $G(B)$ is a group.
\end{prop}

 Let $M$ be a coalgebra. For $g,h\in G(M)$, set
 \begin{align}
 P_{g,h}(M)=\{u \in M \mid \Delta(u) = u\otimes g+ h\otimes u\}.
 \end{align}
 Suppose that $B$ is a bialgebra. Note that $1\in G(B)$.  Define $P(B)=P_{1,1}(B)$:
  \begin{align}
 P(B)=\{u \in B \mid \Delta(u) = u\otimes 1+ 1\otimes u\}.
 \end{align}
 A basic result   (cf. \cite{Abe}) is that for any bialgebra $B$, $P(B)$ is a Lie subalgebra of $B_{\rm Lie}$, 
 where $B_{\rm Lie}$ denotes the Lie algebra with $[a,b]=ab-ba$ for $a,b\in B$.

The following is a well known fact (cf. \cite{M}):%[Proposition 5.5.3]:

\begin{prop}\label{U(g)}
For any Lie algebra $\g$, the universal enveloping algebra $U(\g)$ as a coalgebra is cocommutative and connected
with $P(U(\g))=\g$.
\end{prop}

The following is one of the fundamental results on Hopf algebras 
(see  \cite{Sw}, \cite{M}): % [Corollary 5.6.4, Theorem 5.6.5]): 

\begin{theo}\label{thm of pointed Hopf alg}
Let $H$ be a cocommutative (pointed) Hopf algebra and let $H_1$ be the irreducible connected component containing $1$. 
Then  $H_1$ is a Hopf subalgebra of $H$ such that $P(H_1)=P(H)$ and $H_1\cong U(P(H))$.
Furthermore, $H\cong H_1 \sharp \mathbb{C}[G(H)]$ (a smash product).
\end{theo}

Let $U$ be a vector space (over $\C$) and let $\{v_\lambda\}_{\lambda\in \Lambda}$ be a basis with $\Lambda$ totally ordered.   
Set
\begin{equation}
\mathbb N^\Lambda_{\rm fin}=\{f:\Lambda\rightarrow \mathbb N \mid f(\lambda)=0
\text{ for all but finitely many } \lambda \in \Lambda\}.
\end{equation}
As $\mathbb N$ with the ordinary addition is a semigroup with $0$ as its identity,
 $\mathbb N^\Lambda_{\rm fin}$ is naturally an abelian semigroup with the zero function as its identity:
For $f,g\in \mathbb N^\Lambda_{\rm fin}$, $f+g$ is defined by
$$(f+g)(\lambda)=f(\lambda)+g(\lambda)\quad \text{ for }\lambda\in \Lambda.$$
For $f\in \mathbb N^\Lambda_{\rm fin}$,  set
$$%\operatorname{supp}(f)=\left\{\lambda \in \Lambda \mid f(\lambda)\neq 0\right\},\quad 
|f|= \sum_{\lambda \in \Lambda}f(\lambda),$$
and for $f,g\in \mathbb N^\Lambda_{\rm fin}$,  set
$$\binom{f}{g}= \prod_{\lambda \in \Lambda}\binom{f(\lambda)}{g(\lambda)},$$
where  $\binom{m}{k}=\frac{m\dot(m-1)\cdots(m-k+1)}{k!}$ for $m,k\in \mathbb N$.

The following result can be found in \cite{Sw}: %[Theorem 12.3.2]

\begin{theo} \label{basis of B(V)}
Define ${\mathcal{B}}(U)$ to be the vector space  (over $\C$) with a basis $\{v_{(f)} \mid f \in {\mathbb N}^\Lambda_{\rm fin}\}$.
Then ${\mathcal{B}}(U)$ is a bialgebra with the algebra and coalgebra structure maps given by
\begin{align}\label{d of B(V)}
&v_{(f)}v_{(g)}=\binom{f+g}{f}v_{(f+g)},\quad v_{(0)}=1,\\	
&\Delta(v_{(f)})= \sum_{g+h=f}v_{(g)}\otimes v_{(h)}, 
	\quad \varepsilon(v_{(f)})=\delta_{|f|,0}\quad \text{ for }f,g \in {\mathbb N}^\Lambda_{\rm fin}.
\end{align}
\end{theo}

\begin{remark}
{\em Note that the bialgebra ${\mathcal{B}}(U)$ was originally defined as the (unique) maximal cocommutative subcoalgebra,
namely,  the sum of all cocommutative subcoalgebras, of the ``{\em shuffle algebra}'' $\operatorname{Sh}(U)$. 
As here we only need to use  ${\mathcal{B}}(U)$, we simply introduce this bialgebra this way.}
\end{remark}

Note that ${\mathcal{B}}(U)=\coprod_{i\in \N}{\mathcal{B}}(U)_{(i)}$ is an $\N$-graded bialgebra, where for $i\in \N$, 
\begin{align}
{\mathcal{B}}(U)_{(i)}=\text{span}\{ v_{(f)}\ |\  f \in {\mathbb N}^\Lambda_{\rm fin}\  \text{with }|f|=i\}.
\end{align}
In particular, ${\mathcal{B}}(U)_{(0)}=\C$ and ${\mathcal{B}}(U)_{(1)}=U$. 
Then we have a canonical projection map
\begin{align}
\pi: \  {\mathcal{B}}(U)\rightarrow U.
\end{align}

The following result can be found in \cite{Sw}: %[Page 260]

\begin{prop}\label{shuffle-primitive}
Let $U$ be a vector space. Then  ${\mathcal{B}}(U)$ is  a commutative and cocommutative connected bialgebra with
$P({\mathcal{B}}(U))=U$. On the other hand, ${\mathcal{B}}(U)$ is a cofree  cocommutative coalgebra
in the sense that for any cocommutative coalgebra $C$ together with a linear map $\psi: C\rightarrow U$, there exists a unique 
coalgebra morphism $\Psi:  C\rightarrow {\mathcal{B}}(U)$ such that $\pi\circ \Psi=\psi$.
\end{prop}

Let $M$ be any  connected coalgebra. Then there exists a unique group-like element $e$ and $M^{(0)}=\C e$.
Then $P(M)$ is well defined as 
 \begin{align}
 P(M)=\{u \in M \mid \Delta(u) = u\otimes e+ e\otimes u\}.
 \end{align}
The following result can be found in \cite{Sw}: %[Theorem 12.2.6]

\begin{theo}\label{Important}
Assume that $V$ is a vector space. Let $M$ be any cocommutative connected coalgebra  and
$\psi: {\mathcal{B}}(V) \rightarrow M$ any coalgebra morphism.
Then $\psi$ is injective (resp. surjective, bijective)
whenever $\psi|_{V}: V \rightarrow P(M)$ is injective (resp. surjective, bijective).
\end{theo}

Let  ${\mathfrak{g}}$ be a Lie algebra.  Pick a basis $\{ v_{\lambda}\ |\ \lambda\in \Lambda\}$ for ${\mathfrak{g}}$
with $\Lambda$ totally ordered. Recall that $v_{\lambda}^{(n)}=\frac{1}{n!}v_\lambda^n\in U({\mathfrak{g}})$ for $n\in \N$.
For $f \in {\mathbb N}^{\Lambda}_{\rm fin}$, set 
\begin{align}
v^{(f)}= \prod_{\lambda \in \Lambda}v_{\lambda}^{(f(\lambda))}\in U({\mathfrak{g}}).
\end{align}
 Define a linear map
\begin{align}\label{phi-g}
\Psi_{\mathfrak{g}}:\  {\mathcal{B}}(\mathfrak{g})\rightarrow U({\mathfrak{g}}); \quad 
v_{(f)}\mapsto v^{(f)} \quad \text{ for }f\in \N^{\Lambda}_{\rm fin}.
\end{align}
It follows from the coalgebra structure of $U({\mathfrak{g}})$ and the P-B-W theorem 
that $\Psi_{\mathfrak{g}}$ is a coalgebra isomorphism.
As an immediate consequence of Theorem \ref{Important}, we have:

\begin{coro}\label{B(g)=U(g)}
 Let ${\mathfrak{g}}$ be a Lie algebra and let $M$ be any cocommutative connected coalgebra.
 Suppose $\psi: U({\mathfrak{g}})\rightarrow M$ is a coalgebra morphism. Then $\psi$ is injective (resp. surjective, bijective)
if and only if the restriction $\psi|_{\mathfrak{g}}: {\mathfrak{g}} \rightarrow P(M)$ is injective (resp. surjective, bijective).
\end{coro}

\section{Vertex bialgebras associated to vertex Lie algebras}

In this section, we recall from \cite{L}  the vertex bialgebra ${\mathcal{V}}_{C}$ associated to a vertex Lie algebra 
(namely, a Lie conformal algebra) $C$ and  show that ${\mathcal{V}}_{C}$  as a coalgebra 
is connected (namely, pointed and irreducible) with $P({\mathcal{V}}_{C})=C$.

We begin by recalling the definition of a vertex algebra (see \cite{Bo}, \cite{FLM}, \cite{FHL}; cf. \cite{LL}). 

\begin{defi}\label{def-VA}
\emph{ A \emph{vertex algebra} is a vector space $V$ equipped with a linear map
	$$
	\begin{aligned}
	Y(\cdot, x):\  & V \rightarrow (\operatorname{End}V)[[x, x^{-1}]] \\
	& v \mapsto Y(v, x)=\sum_{n \in \mathbb{Z}} v_{n} x^{-n-1} \quad\left(\text {where } v_{n} \in \operatorname{End}V\right)
	\end{aligned}
	$$
	and  a distinguished vector ${\bf 1}\in V,$ called the \emph {vacuum vector}, satisfying all the following conditions:
	 The \emph{truncation condition}:
	$$Y(u,x)v\in V((x)) \quad\text{for } u,v\in V.$$
The  \emph{vacuum property}:
	$$
	Y({\bf 1}, x)=1.
	$$
The \emph{creation property}:
	$$
	Y(v, x){\bf 1} \in V[[x]] \  \text { and } \lim _{x \rightarrow 0} Y(v, x){\bf 1}=v \quad \text { for } v \in V
	$$
	 and for $u, v \in V,$
\begin{align}\label{Jacobi identity}
x_{0}^{-1} \delta\left(\frac{x_{1}-x_{2}}{x_{0}}\right) &Y(u, x_{1}) Y(v, x_{2})
-x_{0}^{-1} \delta\left(\frac{x_{2}-x_{1}}{-x_{0}}\right) Y(v, x_{2}) Y(u, x_{1}) \\
&=x_{2}^{-1} \delta\left(\frac{x_{1}-x_{0}}{x_{2}}\right) Y(Y(u, x_{0}) v, x_{2})\nonumber
\end{align}
(the \emph{Jacobi identity}), where 
$$x_{0}^{-1} \delta\left(\frac{x_{1}-x_{2}}{x_{0}}\right)=\sum_{n\in \mathbb{Z}} x_0^{-n-1}(x_1-x_2)^n
=\sum_{n\in \mathbb{Z}}\sum_{i\ge 0}\binom{n}{i}(-1)^i x_0^{-n-1}x_1^{n-i}x_2^i.$$
}
\end{defi}

Let $V$ be a vertex algebra. Denote by $\mathcal{D}$ the linear operator on $V$ defined by ${\mathcal{D}}v=v_{-2}{\bf1}$
for $v\in V$.  Then
\begin{align*}
&[{\mathcal{D}},Y(v,x)]=Y({\mathcal{D}}v,x)=\frac{d}{dx}Y(v,x),\\
\label{sew-sytry}&Y(u,x)v=e^{x{\mathcal{D}}}Y(v,-x)u \quad \text{(the {\em skew-symmetry})}
\end{align*}
 for $u,v\in V$. In terms of components,  we have $({\mathcal{D}}u)_m=-m u_{m-1}$ for $m\in\Z$.

\begin{remark}\label{ex-0-add}\rm
Let $u, v\in V$ (a vertex algebra). We have
	\begin{align*}
	[Y(u,x_1), Y(v,x_2)]&={\rm Res}_{x_0} x_{2}^{-1} \delta\left(\frac{x_{1}-x_{0}}{x_{2}}\right) Y(Y(u, x_{0}) v, x_{2})\\
	&=\sum_{j\ge 0}Y(u_jv,x_2)\frac{1}{j!}\left(\frac{\partial}{\partial x_2}\right)^jx_1^{-1}\delta\left(\frac{x_2}{x_1}\right)
	\end{align*}
	(the {\em Borcherds commutator formula}), or equivalently  in the component form,
	\begin{align*}
	[u_m,v_n]=\sum_{j\in \mathbb N}\binom{m}{j}(u_jv)_{m+n-j}
	\end{align*}	for $m,n\in\Z$.
\end{remark}

\begin{exa}\label{Borch-constru}
{\em Let $A$ be a unital commutative associative algebra with a derivation $\pa$. 
A result due to Borcherds (see \cite{Bo}) is that 
$A$ is a vertex algebra with the identity as the vacuum vector and with the vertex operator map $Y(\cdot,x)$ defined by
$$Y(a,x)b=\left(e^{x\partial}a\right)b=\sum_{n\ge 0}\frac{1}{n!}(\partial^na)bx^{n}\quad \text{ for }a,b\in A.$$ 
This construction  of vertex algebras is known as the Borcherds construction.}
\end{exa}

\begin{comment}
\begin{remark}\label{skew-symm}
\rm Let $D$ be the endomorphism of $V$ defined by $Dv=v_{-2}{\bf1}$. Then for any $u,v\in V$, we have $(Du)_m=-m u_{m-1}$ for $m\in\Z$ and  the \emph{skew-symmetry}: $Y(u,x)v=e^{xD}Y(v,-x)u$ (see \cite{LL}).
\end{remark}\end{comment}

Let $V$ and $K$ be vertex algebras. A \emph{vertex algebra (homo)morphism} from $V$ to $K$ is
a linear map $\psi$ such that
		$$
		\begin{aligned}
		\psi({\bf1}) &={\bf1}, \\
		\psi (Y(u, x) v) &=Y(\psi(u), x) \psi(v) \quad \text { for } u, v \in V.
		\end{aligned}
		$$

Recall from  \cite{FHL} that for any vertex algebras $U$ and $V$,  $U\otimes V$ is a vertex algebra
with ${\bf 1}\otimes {\bf 1}$ as the vacuum vector and with
\begin{align}
Y(u\otimes v,x)=Y(u,x)\otimes Y(v,x)\quad \text{  for }u\in U,\ v\in V.
\end{align}

Now, we introduce the main objects of this paper, which are a family of nonlocal vertex bialgebras introduced in \cite{L}.

\begin{defi}
{\em A {\it vertex bialgebra} is a  vertex algebra $V$ equipped with a coalgebra structure
such that the coalgebra structure maps $\Delta: V \rightarrow V \otimes V$ and $\varepsilon: V \rightarrow \mathbb{C}$ are  vertex algebra morphisms.}
\end{defi}

A vertex bialgebra is said to be {\it cocommutative} (resp. {\it pointed, irreducible, connected})
if it  is cocommutative (resp. pointed, irreducible, connected) as a coalgebra.
A {\it vertex bialgebra morphism} is  both a vertex algebra morphism and a coalgebra morphism.

Let $(A,\pa)$ be a differential algebra where $A$ is a unital commutative associative algebra 
and  $\pa$ is a  derivation of $A$. 
Then $A\otimes A$ and $\mathbb{C}$ are naturally differential algebras with 
  $\pa_{A\otimes A}=\pa\otimes 1+1\otimes\pa$ and $\pa_\mathbb{C}=0$. 
 %We shall write $(A, \pa)$ as $A$ for short if the context makes the situation clear.   
A \emph{differential algebra morphism} is  an algebra morphism $f: A \rightarrow B$ such that $f\pa_A=\pa_B f$.
A \emph{differential bialgebra} is a bialgebra  $B$ equipped with a derivation $\pa$ of $B$ as an algebra such that
 \begin{align}
 \varepsilon\pa=0\ (=\pa \varepsilon),\quad \Delta\pa
 =(\pa\otimes 1+1\otimes\pa)\Delta,
 \end{align}
 namely, $\Delta$ and $\varepsilon$ are differential algebra morphisms.  
 A \emph{differential bialgebra morphism}   is defined to be a differential algebra morphism which is also a coalgebra morphism.
 
 \begin{remark}\label{vba-dba}
 {\em Recall that every commutative differential algebra $A$ is naturally a vertex algebra 
 by Borcherds' construction (see Example \ref{Borch-constru}).
 On the other hand, as it was shown in \cite{L}, %[Example 4.2], 
 every differential bialgebra is naturally a vertex bialgebra.}
 \end{remark}

Next, we discuss vertex Lie algebras. A notion of vertex Lie algebra was introduced by Primc (see \cite{P}) 
and an equivalent notion, called Lie conformal algebra, was independently introduced by Kac (see  \cite{Kac}).
(Another essentially equivalent notion of vertex Lie algebra was studied in  \cite{DLM}.)
Here, we shall use Primc's definition. 

\begin{defi} \label{def-vla}
{\em A {\em vertex Lie algebra} is a vector space $U$ equipped with a linear operator $\mathcal{D}$ and a linear map
\begin{align*}
Y^{-}(\cdot,x): \ U &\rightarrow x^{-1}({\rm End}\,U)[[x^{-1}]]\nonumber\\
u&\mapsto Y^{-}(u,x)=\sum_{n\in \mathbb N}u_nx^{-n-1}\quad (\text{where }u_n\in {\rm End}\,U),
\end{align*}
satisfying the following conditions for $u,v,w\in U,\ m,n\in \mathbb N$:
\begin{align*}
&u_nv=0\quad\text{ for  $n$ sufficiently large,}\\
&(\mathcal{D} u)_{n}v=-nu_{n-1}v, \quad u_n\mathcal{D} v=\mathcal{D}(u_nv)-nu_{n-1}v,\\
&u_nv=\sum_{j\ge 0}(-1)^{n+j+1}\frac{1}{j!}\mathcal{D}^{j}v_{n+j}u,\\
&u_mv_nw-v_nu_mw=\sum_{j\ge 0}\binom{m}{j}(u_jv)_{m+n-j}w.
\end{align*}}
\end{defi}

Note that every vertex algebra $V$ is naturally a vertex Lie algebra with 
$Y^{-}(u,x)=\sum_{n\ge 0}u_nx^{-n-1}$, the singular part of the vertex operator $Y(u,x)$ for $u\in V$.

Let $U$ and $V$ be two vertex Lie algebras. 
A {\em vertex Lie subalgebra} of $U$ is a subspace $W$ such that $\mathcal D(W)\subset W$ and $u_mv\in W$ 
for any $u,v\in W,\  m\in\mathbb N$.  A {\em vertex Lie algebra morphism}  from $U$ to $V$ 
is a linear map $f$  such that 
\begin{align}
f(Y_U^{-}(u,x)v)=Y_V^{-}(f(u),x)f(v)\  \  \text{ and }\  f(\mathcal D_Uu)=\mathcal D_V(f(u))\quad \text{ for }u, v\in U.
\end{align}

Let $C$ be a vertex Lie algebra. Set $L(C)=C\otimes \C[t,t^{-1}]$, a vector space, and set
$\hat{\mathcal D}=\mathcal D \otimes 1+1 \otimes \frac{d}{dt}$, a linear operator on $L(C)$.
  We have a Lie algebra
  $\mathcal{L}_{C}$  (see \cite{P}), where
 \begin{align}
 \mathcal{L}_{C}=L(C)/\hat{\mathcal D}L(C)
 \end{align}
 as a vector space and the Lie bracket is given by
\begin{align}
[a(m),b(n)]= \sum_{j \in \mathbb N} \dbinom{m}{j}(a_{j}b)(m+n-j)
\end{align}
for $a,b\in C,\ m,n\in \mathbb Z$, where $u(r)$ denotes the image of $u\otimes t^r$  in $ \mathcal{L}_{C}$
for $u\in C,\ r\in\mathbb Z$.

Lie algebra $ \mathcal{L}_{C}$ has a polar decomposition $ \mathcal{L}_{C}= \mathcal{L}_{C}^+ \oplus \mathcal{L}_{C}^-$ 
into Lie subalgebras, where
\begin{align}
\mathcal{L}_{C}^+={\rm span}\{ a(n)\,|\, a \in C,\ n\geq 0\},\quad
\mathcal{L}_{C}^-= {\rm span}\{ a(n)\,|\, a \in C,\  n< 0\}.
\end{align}
View  $C$ as a trivial $\mathcal{L}_{C}^+$-module and then form the induced $\mathcal{L}_{C}$-module:
\begin{equation}\label{Vc}
{\mathcal{V}}_{C}=U\left(\mathcal{L}_{C}\right) \otimes_{U\left(\mathcal{L}_{C}^{+}\right)} \mathbb{C}.
\end{equation}
Set ${\bf 1}=1\otimes 1\in {\mathcal{V}}_{C}$. By the P-B-W theorem, the natural $U(\mathcal{L}_{C}^{-})$-module
homomorphism from $U(\mathcal{L}_{C}^{-})$ to ${\mathcal{V}}_{C}$, sending $X$ to $X{\bf 1}$ for $X\in U(\mathcal{L}_{C}^{-})$, 
is an isomorphism.
Thus $U(\mathcal{L}_{C}^{-})\simeq {\mathcal{V}}_{C}$ as a $U(\mathcal{L}_{C}^{-})$-module.
We view $C$ as a subspace of ${\mathcal{V}}_{C}$ by identifying $a$ with $a(-1){\bf 1}$ for $a\in C$.

The following result can be found in \cite{DLM,P}:

\begin{prop} \label{3.7}
	On the $\mathcal{L}_{{C}}$-module $\mathcal{V}_{C}$,
	there exists a vertex algebra structure which is uniquely determined by the condition that
	${\bf 1}$ is the vacuum vector and
	$$
	Y(u, x)=\sum_{n \in \mathbb{Z}} u(n) x^{-n-1} \quad \text { for } u \in C.
	$$
Furthermore, ${\mathcal{V}}_{{C}}$ is generated by ${C}$.
\end{prop}

The following result was obtained in \cite{L}:

\begin{prop}\label{3.8}
Let $C$ be a vertex Lie algebra.  Then there exist  vertex algebra morphisms
 $\Delta: {\mathcal{V}}_{C} \rightarrow {\mathcal{V}}_{C}\otimes {\mathcal{V}}_{C}$ and 
 $\varepsilon: {\mathcal{V}}_{C} \rightarrow \mathbb{C}$, 
 which are uniquely determined by
\begin{align*}
	&\Delta({\bf 1})= {\bf 1}\otimes {\bf 1},  \quad \varepsilon({\bf 1})= 1,\\
	& \Delta(a)= a\otimes {\bf1}+ {\bf1} \otimes a, \quad \varepsilon(a)=0\quad \text{  for }a \in {C}.
\end{align*}
Furthermore, ${\mathcal{V}}_{C}$ equipped with $\Delta$ and $\varepsilon$ is a cocommutative  vertex bialgebra.
\end{prop}

Let  $C$ be a vertex Lie algebra as before. It is straightforward to show that the linear map 
${C}\rightarrow \mathcal{L}_{C};\  a\mapsto a(-1)$ is one-to-one.
Then we view ${C}$ as a subspace of $\mathcal{L}_{C}$
by identifying $a$ with  $a(-1)$ for $a\in C$.
Recall 
$\mathcal{L}_\mathcal C^-={\rm span}\{ a(-n-1) \,|\, a\in {C},\ n\in \N\}.$
Noticing  that 
\begin{align}
a(-n-1)=\frac{1}{n!}(\mathcal D^n a)(-1)\quad \text{for }a\in {C},\ n\in \N,
\end{align}
we have 
\begin{align}
\mathcal{L}_{C}^-=\{ u(-1)\ |\ u\in {C}\}
\end{align}
and $\mathcal{L}_{C}^-={C}$ as vector spaces.

\begin{prop}\label{VC-facts}
Let $C$ be any vertex Lie  algebra. 
Then the $U(\mathcal{L}_{C}^-)$-module isomorphism $\eta: U(\mathcal{L}_{C}^-)\rightarrow {\mathcal{V}}_C$,
defined by 
\begin{align}
\eta(X)=X{\bf 1}\quad \text{for }X\in U(\mathcal{L}_{C}^-),
\end{align}
is a coalgebra isomorphism.  Furthermore,  the associated vertex bialgebra ${\mathcal{V}}_C$ is cocommutative and connected
with $P({\mathcal{V}}_{C})={C}$.
\end{prop}

\begin{proof} The furthermore assertion follows immediately from the first assertion. 
Note that $U(\mathcal{L}_{C}^-)$ as an algebra is generated by $\mathcal{L}_{C}^-=C$. 
Since $\Delta_{U(\mathcal{L}_{C}^-)},$  $\varepsilon_{U(\mathcal{L}_{C}^-)}$ are algebra morphisms, 
and $\Delta_{{\mathcal{V}}_{C}}$,  $\varepsilon_{{\mathcal{V}}_{C}}$ are vertex algebra morphisms,
it follows that $\eta$ is a coalgebra morphism, and hence it is a coalgebra isomorphism.
\end{proof}

\begin{remark}\label{DV}
{\em Let $V$ be a vertex bialgebra. Note that 
\begin{align*}
&\Delta {\mathcal{D}}(v)=({\mathcal{D}}\otimes 1+1\otimes {\mathcal{D}})\Delta(v)
\in {\mathcal{D}}(V)\otimes V+V\otimes {\mathcal{D}}(V),\\
&\varepsilon({\mathcal{D}}v)={\mathcal{D}}\varepsilon(v)=0\quad\text{  for }v\in V.
\end{align*}
Then ${\mathcal{D}}(V)$ is a coideal of $V$, by which we have a factor coalgebra $V/{\mathcal{D}}(V)$. 
On the other hand, it is known (see \cite{Bo}) that $V/{\mathcal{D}}(V)$ is a Lie algebra with
\begin{align*}
[u+{\mathcal{D}}(V), v+{\mathcal{D}}(V)]=u_0v+{\mathcal{D}}(V)\quad\text{  for }u, v\in V.
\end{align*}
Note that there is no canonical Lie algebra structure on $V/{\mathcal{D}}(V)\otimes V/{\mathcal{D}}(V)$.
We know that $(V\otimes V)/({\mathcal{D}}\otimes 1+1\otimes {\mathcal{D}})(V\otimes V)$ is a Lie algebra with
\begin{align}
[\overline{a\otimes b},\overline{u\otimes v}]
=\sum_{m\in \Z}\overline{a_mu\otimes b_{-m-1}v}\quad \text{ for }a,b,u,v\in V.
\end{align}
We have
$$\Delta {\mathcal{D}} (V)=({\mathcal{D}}\otimes 1+1\otimes {\mathcal{D}})\Delta(V)\subset ({\mathcal{D}}\otimes 1+1\otimes {\mathcal{D}})(V\otimes V),$$
but $({\mathcal{D}}\otimes 1+1\otimes {\mathcal{D}})\Delta(V)\ne ({\mathcal{D}}\otimes 1+1\otimes {\mathcal{D}})(V\otimes V)$ 
in general. On the other hand, the Lie algebra bracket on the Lie algebra $V/{\mathcal{D}}(V)$ is not a coalgebra morphism. }
\end{remark}

\begin{remark}\label{zhu-C2}
{\em Recall that a commutative Poisson algebra is a commutative associative algebra $A$
equipped with a Lie algebra structure $\{\cdot,\cdot\}$ such that
$$\{ ab,c\}=a\{b,c\}+\{a,c\}b\quad \text{for }a,b,c\in A.$$
Let $V$ be a vertex algebra. Set 
$$C_2(V)=\text{span}\{ u_{-2}v\ |\ u,v\in V\}.$$ 
From \cite{Zhu1, Zhu}, $V/C_2(V)$ is a commutative Poisson algebra with
$$(u+C_2(V))(v+C_2(V))=u_{-1}v+C_2(V), \ \   \{ u+C_2(V), v+C_2(V)\}=u_{0}v+C_2(V)$$
for $u,v\in V$. }
\end{remark}

\begin{prop}\label{C-2-V-structure}
Let $V$ be a vertex bialgebra. Then $C_2(V)$ is a coideal of $V$ and $V/C_2(V)$ is a bialgebra. 
Furthermore,   $(V/C_2(V))\otimes (V/C_2(V))$ is a  Lie algebra with 
\begin{align}\label{tensor-Lie-algebra}
[\overline{a\otimes b},\overline{u\otimes v}]=\overline{a_{-1}u}\otimes \overline{b_0v}+\overline{a_0u}\otimes \overline{b_{-1}v}
\end{align}
for $a,b,u,v\in V$ and the comultiplication map $\overline{\Delta}$ of $V/C_2(V)$ is a Lie algebra morphism.
\end{prop}

\begin{proof} For $u,v\in V$, we have
\begin{align*}
\Delta(Y(u,x)v)=\sum Y(u^{(1)},x)v^{(1)}\otimes Y(u^{(2)},x)v^{(2)},
\end{align*}
or in terms of components, 
\begin{align}\label{Delta-u-n-v}
\Delta(u_nv)=\sum_{m\in \Z} u^{(1)}_mv^{(1)}\otimes u^{(2)}_{n-m-1}v^{(2)}\quad \text{ for }n\in \Z. 
\end{align}
Recall from \cite{Zhu1, Zhu} that $a_{-m-2}b\in C_2(V)$ for $a,b\in V,\ m\in \N$. 
Then 
\begin{align*}
\Delta(u_{-2}v)=\sum_{m\in \Z}u^{(1)}_{m}v^{(1)}\otimes u^{(2)}_{-m-3}v^{(2)}\subset V\otimes C_2(V)+C_2(V)\otimes V.
\end{align*}
We also have $\varepsilon(u_{-2}v)=0$ for $u,v\in V$, 
recalling that $\varepsilon(u_mv)=\delta_{m,-1}\varepsilon(u)\varepsilon(v)$ for $m\in \Z$.
Thus $C_2(V)$ is a coideal of $V$. Now, the commutative Poisson algebra $V/C_2(V)$ is also a coalgebra.
It follows from (\ref{Delta-u-n-v}) that $V/C_2(V)$ is a bialgebra. 
On the other hand,  it is straightforward to show that 
$$C_2(V\otimes V)=V\otimes C_2(V)+C_2(V)\otimes V.$$
As $\Delta: V\rightarrow V\otimes V$ is a vertex algebra morphism, the factor comultiplication map
$$\overline{\Delta}: V/C_2(V)\rightarrow (V\otimes V)/C_2(V\otimes V)=V/C_2(V)\otimes V/C_2(V)$$
is a Poisson algebra morphism.
Note that the Lie algebra structure on $(V\otimes V)/C_2(V\otimes V)$ is given by
\begin{align}
[a\otimes b+C_2(V\otimes V),u\otimes v+C_2(V\otimes V)]=a_{-1}u\otimes b_0v+a_0u\otimes b_{-1}v+C_2(V\otimes V)
\end{align}
for $a,b,u,v\in V$. 
% in view of Remark \ref{zhu-C2} and Lemma \ref{XYZ},  
Then $V/C_2(V)\otimes V/C_2(V)$ is a  Lie algebra with the Lie bracket defined by (\ref{tensor-Lie-algebra}).
 \end{proof}
 
\begin{remark}\label{C-n-V}
{\em  Let $V$ be a vertex bialgebra. For $n\ge 1$, set  
 $$C_n(V)=\text{span}\{ u_{-n}v\ |\ u,v\in V\}.$$  
 This gives a descending sequence of subspaces with $C_1(V)=V$ such that
 \begin{align}
 \Delta (C_n(V))\subset \sum_{i=1}^n C_i(V)\otimes C_{n+1-i}(V)
 \end{align}
 for $n\ge 1$.}
\end{remark}

\section{Structure of cocommutative vertex bialgebras}
 This section is the core of the whole paper. In this section, 
 first we determine the structures of $G(V)$ and $P(V)$ for a general vertex bialgebra $V$, and then 
 we determine the structure of a general cocommutative vertex bialgebra $V$ such that $G(V)$ is a group 
 which lies in the center of $V$. The main results are Theorem \ref{main-result1}, Corollary \ref{vla-identification}, 
 Theorems \ref{thm-second} and \ref{mixed-vla-group}.  
 
We begin with the structure of the set $G(V)$ of group-like elements.

\begin{prop}\label{group-like-group}
Let $V$ be a vertex bialgebra.
Then
\begin{align}\label{semigp}
u_nv=0,\  \  u_{-1}v\in G(V)\quad \text{for }u,v\in G(V),\ n\ge 0.
\end{align}
Furthermore, $G(V)$ with the operation defined by $gh=g_{-1}h$ for $g,h\in G(V)$
is an abelian semigroup with the vacuum vector ${\bf 1}$ as its identity.
\end{prop}

\begin{proof} Let $u,v\in G(V)$.  We now prove $u_nv=0$ for all $n\ge 0$. Suppose that $u_nv\ne 0$ for some nonnegative integer $n$.
Then there exists a nonnegative integer $k$ such that $u_kv\ne 0$ and $u_mv=0$ for $m>k$.
As $\Delta$ is a vertex algebra homomorphism, we have
\begin{align*}
0&=\Delta(u_{2k+1}v)=\Delta(u)_{2k+1}\Delta(v)={\rm Res}_x x^{2k+1}Y(u\otimes u,x)(v\otimes v)\\
&={\rm Res}_x x^{2k+1}(Y(u,x)v\otimes Y(u,x)v)
=\sum_{m\in \mathbb Z}u_mv\otimes u_{2k-m}v=u_kv\otimes u_kv,
\end{align*}
which contradicts that $u_kv\ne 0$. Thus $u_nv=0$ for all $n\ge 0$. Using this property, we get
\begin{align*}
\Delta(u_{-1}v)&=\Delta(u)_{-1}\Delta(v)={\rm Res}_x x^{-1}(Y(u,x)v\otimes Y(u,x)v)\\
&=\sum_{m\in \mathbb Z}u_mv\otimes u_{-m-2}v=u_{-1}v\otimes u_{-1}v.
\end{align*}
We also have $\varepsilon(u_{-1}v)=\varepsilon(u)_{-1}\varepsilon(v)=1$.
 Therefore, $u_{-1}v\in G(V)$.

Now, let $u,v,w\in G(V)$. We have $u_nv=0=u_nw$ for $n\ge 0$. With $V$ a vertex algebra, from the Jacobi identity
for the triple $(u,v,w)$ we get
\begin{align*}
Y(u,x_1)Y(v,x_2)w&=Y(v,x_2)Y(u,x_1)w,\\
Y(u,x_0+x_2)Y(v,x_2)w&=Y(Y(u,x_0)v,x_2)w,
\end{align*}
which give $u_{-1}v_{-1}w=v_{-1}u_{-1}w$ and $u_{-1}v_{-1}w=(u_{-1}v)_{-1}w$.
It then follows that $G(V)$ is an abelian semigroup with ${\bf 1}$ the identity.
\end{proof}

\begin{remark}
\rm Assume that $V$ is a nonlocal vertex bialgebra. From the same argument we see that  
$G(V)$ is a (not necessarily abelian) semigroup with identity ${\bf 1}$.
\end{remark}

\begin{remark}\label{G(V)-indep}
{\em  Note that for any coalgebra $C$, $G(C)$ is a linearly independent subset.
Then for any vertex bialgebra $V$, $G(V)$ is a linearly independent subset of $V$, so that $\C[G(V)]$ is naturally a subcoalgebra of $V$.}
\end{remark}

\begin{remark}\label{vertex-hopf-algebra}
{\em Recall that a cocommutative (pointed) bialgebra $B$ (over $\C$) is a Hopf algebra if and only if $G(B)$  is a group. 
In view of this, one can define a vertex Hopf algebra as a vertex bialgebra $V$ such that $G(V)$ is a group.}
\end{remark}

\begin{lemm}\label{subalgebra-G(V)}
Let $V$ be a vertex bialgebra.  Denote by $\langle G(V)\rangle$ the vertex subalgebra of $V$, generated by $G(V)$. Then 
$\langle G(V)\rangle$ is a commutative and cocommutative vertex bialgebra. 
\end{lemm}

\begin{proof} From the proof of Proposition \ref{group-like-group},
we have $[Y(g,x_1),Y(h,x_2)]=0$ for $g,h\in G(V)$. Then it follows that $\langle G(V)\rangle$ is a commutative vertex subalgebra.
Set
$$K=\{ v\in V\ |\  \Delta(v)\subset \langle G(V)\rangle\otimes \langle G(V)\rangle\}.$$
It is clear that $G(V)\subset K$, in particular ${\bf 1}\in K$. Let $u,v\in K$. Then
$$\Delta(u)=\sum u^{(1)}\otimes u^{(2)},\  \  \Delta(v)=\sum v^{(1)}\otimes v^{(2)}\in \langle G(V)\rangle\otimes \langle G(V)\rangle, $$
from which we get
$$\Delta(Y(u,x)v)=\sum Y(u^{(1)},x)v^{(1)}\otimes Y(u^{(2)},x)v^{(2)}\in  (\langle G(V)\rangle\otimes \langle G(V)\rangle)((x)).$$
Thus $K$ is a vertex subalgebra. Consequently, $\langle G(V)\rangle\subset K$. Therefore,
$\langle G(V)\rangle$ is a subcoalgebra. To show that it is cocommutative, set
$$E=\{ v\in V\ |\ \Delta(v)=T\Delta(v)\},$$
where $T$ denotes the flip operator on $V\otimes V$. It is straightforward to show that $E$ is a vertex subalgebra of $V$, containing $G(V)$.
Consequently, we have $\langle G(V)\rangle\subset E$, proving that $\langle G(V)\rangle$ is cocommutative.
 \end{proof}

\begin{remark}
{\rm Note that the vertex subbialgebra $\langle G(V)\rangle$  is isomorphic to the vertex bialgebra obtained from 
a cocommutative (and commutative) differential bialgebra through Borcherds' construction.}
\end{remark}

\begin{remark}\label{LA-on-va}
{\em Let $C$ be a vertex Lie algebra. Recall that the correspondence 
$C\ni a\mapsto a(-1)\in {\mathcal{L}}_C$ is a linear bijection between $C$ and the Lie subalgebra ${\mathcal{L}}_C^{-}$.
This gives rise to a Lie algebra structure on $C$. Note that
$$[a(-1),b(-1)]=\sum_{i\ge 0}\binom{-1}{i} (a_ib)(-2-i)=\sum_{i\ge 0}(-1)^i\frac{1}{(i+1)!}\left({\mathcal{D}}^{i+1}a_ib\right)(-1)$$
in ${\mathcal{L}}_C^{-}$ for $a,b\in C$. Then
\begin{align}
[a,b]=\sum_{i\ge 0}(-1)^i\frac{1}{(i+1)!}{\mathcal{D}}^{i+1}a_ib\quad \text{ in }C.
\end{align}}
 \end{remark}

Recall that every vertex algebra is naturally a vertex Lie algebra and 
that for a vertex bialgebra $V$, $P(V)=\{ v\in V\ |\  \Delta{v}=v\otimes {\bf 1}+{\bf 1}\otimes v\}$.
Also recall that for a vertex algebra $V$, ${\mathcal{D}}$ is the linear operator on $V$, 
defined by 
$${\mathcal{D}}(v)=v_{-2}{\bf 1}=\text{Res}_x x^{-2}Y(v,x){\bf 1}\quad \text{  for }v\in V.$$
 Note that for $v\in V$, ${\mathcal{D}}v=0$ if and only if $v_{n}=0$ for all $n\ne -1$.

\begin{prop}\label{va-P(V)}
 Let $V$ be a vertex bialgebra. Then $P(V)$ is a vertex Lie subalgebra of $V$ with $\varepsilon(v)=0$ for $v\in P(V)$, 
 and $P(V)$ is a $\C[{\mathcal{D}}]$-submodule. 
 Furthermore, $V$ is an ${\mathcal{L}}_{P(V)}$-module with $a(n)$ acting as $a_n$ for $a\in P(V),\ n\in \Z$.
\end{prop}

\begin{proof}  For $a\in P(V)$, with $\Delta({\bf 1})={\bf 1}\otimes {\bf 1}$ we have
  \begin{eqnarray*}
  \Delta(\mathcal D a)\!\!\!&=&\!\!\!\Delta(a_{-2}{\bf1})={\rm Res}_x x^{-2} \Delta (Y(a,x){\bf 1})={\rm Res}_x x^{-2}Y(\Delta(a),x)\Delta({\bf1})\\
\!\!\!&=&\!\!\!{\rm Res}_x x^{-2} \big(Y(a,x){\bf1}\otimes {\bf 1}+  {\bf 1}\otimes Y(a,x){\bf1}\big)
=\mathcal Da\otimes {\bf 1}+{\bf 1}\otimes \mathcal Da,
\end{eqnarray*}
which implies $\mathcal D a\in P(V)$. Thus, $\mathcal D(P(V))\subset P(V)$.
Now, let $a,b \in P(V)$. Then
	\begin{align*}
	\Delta(Y(a,x)b)&=Y(\Delta(a),x)\Delta(b)\\
	&=\big( Y(a,x)\otimes 1+ 1\otimes Y(a,x)\big)(b\otimes {\bf 1} +{\bf 1}\otimes b)\\
	&=Y(a,x)b\otimes{\bf 1}+b\otimes Y(a,x){\bf 1}+Y(a,x){\bf 1}\otimes b+{\bf 1}\otimes Y(a,x)b.
	\end{align*}
For $m\in \mathbb{N}$, as $a_m{\bf 1}=0$ we get
$$\Delta(a_mb)=a_mb\otimes {\bf 1}+{\bf 1}\otimes a_mb,$$
so $a_mb\in P(V)$. Therefore, $P(V)$ is a vertex Lie subalgebra of $V$.
On the other hand,  for $v \in P(V)$ we have
$$v= (1\otimes \varepsilon)\Delta(v)=(1\otimes \varepsilon)(v\otimes {\bf 1}+ {\bf 1} \otimes v)
	= \varepsilon ({\bf 1})v+ \varepsilon(v){\bf 1}=v+ \varepsilon(v){\bf 1},$$
which implies $\varepsilon(v)=0$. The furthermore assertion follows immediately from the construction of Lie algebra ${\mathcal{L}}_{P(V)}$.
\end{proof}

\begin{lemm}\label{va-coalgebra-morphism}
Let $V$ be a vertex bialgebra. Define  a linear map 
\begin{align}
\phi:\  U({\mathcal {L}}_{P(V)})\otimes V\rightarrow V;\quad 
X\otimes v\mapsto \phi(X\otimes v)=Xv
\end{align}
for $X\in U({\mathcal {L}}_{P(V)}),\ v\in V$. 
Then  $\phi$ is a coalgebra morphism.
\end{lemm} 

\begin{proof} For clarity, let $\Delta^U$, $\Delta^V$, and $\Delta^t$ denote the comultiplication maps of 
$U({\mathcal {L}}_{P(V)})$, $V$, and $U({\mathcal {L}}_{P(V)})\otimes V$, respectively.
Set
$$E=\{ X\in U({\mathcal {L}}_{P(V)})\ |\  \Delta^V(\phi(X\otimes v))=(\phi\otimes \phi)\Delta^t(X\otimes v)\  \  \text{for all }v\in V\}.$$
Note that  $\Delta^t=T^{23}(\Delta^U\otimes \Delta^V)$,
where $T^{23}$ denotes the flip operator with respect to the indicated factors on the tensor product space (with four factors).
It is clear that $1\in E$.
Let $a\in P(V),\ n\in \Z,\ X\in E\subset U({\mathcal {L}}_{P(V)})$. For any $v\in V$, we have
\begin{align*}
&\Delta^V(\phi(a_nX\otimes v))=\Delta^V(a_nXv)=\Delta^V(a)_n\Delta^V(Xv)=(a_n\otimes 1+1\otimes a_n)\Delta^V(Xv)\nonumber\\
&\quad =(a_n\otimes 1+1\otimes a_n)(\phi\otimes \phi)T^{23}(\Delta^U(X)\otimes \Delta^V(v))\nonumber\\
&\quad =(\phi\otimes \phi)T^{23}(\Delta^U(a_n)\Delta^U(X)\otimes \Delta^V(v))\nonumber\\
&\quad =(\phi\otimes \phi)T^{23}(\Delta^U(a_nX)\otimes \Delta^V(v))\\
&\quad =(\phi\otimes \phi)\Delta^t(a_nX\otimes v).
\end{align*}
Thus $a_nX\in E$. It follows that $E=U({\mathcal {L}}_{P(V)})$, so that $\Delta^V \phi=(\phi\otimes \phi)\Delta^t$. 
On the other hand, we have $\varepsilon_V (\phi(1\otimes v))=\varepsilon_V (v)=\varepsilon_t (1\otimes v)$ for $v\in V$, and 
$$\varepsilon_V (\phi(a_nX\otimes v))=\varepsilon_V (a_nXv)=\delta_{n,-1}\varepsilon_V(a)\varepsilon_V(Xv)=0,$$
$$
\varepsilon_t(a_nX\otimes v)=\varepsilon_U(a_nX)\varepsilon_V(v)=\varepsilon_U(a_n)\varepsilon_U(X)\varepsilon_V(v)=0$$
for any $a\in P(V),\ n\in \Z,\ X\in U({\mathcal {L}}_{P(V)}),\ v\in V$, noticing that $\varepsilon_V(a)=0=\varepsilon_U(a_n)$. 
It then follows that $\varepsilon_V \circ \phi=\varepsilon_t$.
Therefore, $\phi$ is a coalgebra morphism.
\end{proof}

Similarly, we have the following straightforward result:

\begin{lemm}\label{D-coalgebra-morphism}
Let $V$ be a vertex bialgebra. Define  a linear map 
\begin{align}
\phi_{\mathcal{D}}:\  \C[{\mathcal {D}}]\otimes V\rightarrow V;\quad 
{\mathcal{D}}^n\otimes v\mapsto {\mathcal{D}}^nv
\end{align}
for $n\in \N,\ v\in V$. View $\C[{\mathcal {D}}]$ as a coalgebra by identifying $\C[\mathcal{D}]$ 
as the universal enveloping algebra of the $1$-dimensional Lie algebra $\C {\mathcal{D}}$.
Then  $\phi_{\mathcal{D}}$ is a coalgebra morphism.
\end{lemm}

\begin{defi}\label{Vg}
{\em Let $V$ be a vertex bialgebra. For $g\in G(V)$, let $V_g$ denote the irreducible component of $V$, containing $\C g$
(an irreducible (simple) subcoalgebra).}
\end{defi}

The following is a technical result:

\begin{lemm}\label{PggV-PVg}
Let $V$ be a vertex bialgebra. Then 
\begin{align}
{\mathcal{L}}_{P(V)}g\subset P_{g,g}(V)\quad \text{ for }g\in G(V).
\end{align}
If $V$ is cocommutative, we have
\begin{align}
P_{g,g}(V)=P(V_g)\quad \text{ for }g\in G(V).
\end{align}
\end{lemm}

\begin{proof} Let $g\in G(V),\ a\in P(V),\ n\in \Z$. We have
$$\Delta(a_ng)=\Delta(a)_n\Delta(g)=(a_n\otimes 1+1\otimes a_n)(g\otimes g)=a_ng\otimes g+g\otimes a_ng,$$
which implies that $a_ng\in P_{g,g}(V)$. This proves the first assertion.
Assume that $V$ is cocommutative. Then $V=\oplus_{g\in G(V)}V_g$.
Let $v\in P_{g,g}(V)$ with $g\in G(V)$. Write $v=\sum_{h\in G(V)}v_{h}$ (a finite sum) with $v_h\in V_h$. 
We have
\begin{align*}
&\Delta(v)=\sum_{h\in G(V)}\Delta(v_{h})\in \bigoplus_{h\in G(V)}V_h\otimes V_h,\\
&\Delta(v)=g\otimes v+v\otimes g=\sum_{h\in G(V)}(v_{h}\otimes g+g\otimes v_h),
\end{align*}
which imply $v=v_g\in V_g$. Thus $P_{g,g}(V)=P_{g,g}(V_g)=P(V_g)$.
\end{proof}

Note that   if $V$ is an irreducible vertex bialgebra, then $G(V)=\{ {\bf 1}\}$ and $V=V_{\bf 1}$.
On the other hand, if $V$ is cocommutative (hence pointed),  we have $V=\bigoplus_{g\in G(V)}V_g$.

As the first main result of this paper, we have:

\begin{theo}\label{main-result1}
Let $V$ be a cocommutative vertex bialgebra. 
Denote by $\langle P(V)\rangle$ the vertex subalgebra of $V$ generated by $P(V)$.
Then in the connected component decomposition
\begin{align}
V=\bigoplus_{g\in G(V)}V_g,
\end{align}
$V_{\bf 1}$ is a vertex subbialgebra of $V$ and $V_{\bf 1}=\langle P(V)\rangle\simeq {\mathcal{V}}_{P(V)}$.
Furthermore, for $g\in G(V)$, $V_g$ is a submodule of $V$ for $V_{\bf 1}$ as a vertex algebra. 
\end{theo}

\begin{proof} For convenience, set $U=\langle P(V)\rangle$. As a subspace,
$U$ is linearly spanned by vectors
$$a_{n_1}^{(1)}\cdots a^{(r)}_{n_r}{\bf 1}$$
for $r\in \N,\ a^{(i)}\in P(V),\ n_i\in \Z$. That is, $U=U({\mathcal{L}}_{P(V)}){\bf 1}$. With $P(V)$ a vertex Lie subalgebra of $V$,
by a result of Primc (see \cite{P}), there exists a vertex algebra morphism $\psi:  {\mathcal{V}}_{P(V)} \rightarrow V$ 
with $\psi|_{P(V)}=1$, where ${\mathcal{V}}_{P(V)}$ is
the vertex algebra  associated to the vertex Lie algebra $P(V)$. It is clear that $\psi({\mathcal{V}}_{P(V)})=U$.
%On the other hand, for $a\in P(V)$,  $v\in V$,  we have
%$$\Delta(Y(a,x)v)=Y(\Delta(a),x)\Delta(v)=(Y(a,x)\otimes 1+1\otimes Y(a,x))\Delta(v).$$
%Thus
%$$\Delta(a)_{n}=a_{n}\otimes 1+1\otimes a_n\quad \text{ for }n\in \Z.$$
By using the fact that $\Delta$ is a vertex algebra morphism, it is straightforward to show that
 the subset $M:=\{ v\in V\ |\ \Delta(v)\in U\otimes U\}$ is a vertex subalgebra containing $P(V)$.
Consequently, we have $U\subset M$, and hence $U$ is a subcoalgebra of $V$.
Similarly, by showing that the subset
$$W:=\{ v\in \mathcal{V}_{P(V)}\ |\  (\psi\otimes \psi)\Delta(v)=\Delta \psi(v),\  \varepsilon \psi(v)=\varepsilon(v)\}$$
is a vertex subalgebra containing $P(V)$,
we conclude that $W=\mathcal{V}_{P(V)}$. Consequently, $\psi$ is also a coalgebra morphism.
Since ${\mathcal{V}}_{P(V)}$ is cocommutative and connected,
 in view of Theorem \ref{Abe-C2.4.21-T2.4.22} $U$ is irreducible (and connected) with ${\bf 1}\in U$. Then $U\subset V_{\bf 1}$.
From this we have $P(V)\subset U\subset V_{\bf 1}$. It follows that $P(V)=P(V_{\bf 1})$.

View $P(V)$ as a Lie algebra as in Remark \ref{LA-on-va}.  
Recall that we have a canonical coalgebra isomorphism
$$\eta: \  U(P(V))\rightarrow  {\mathcal V}_{P(V)}.$$ 
Combining $\eta$ with the coalgebra morphism $\psi$, we get a coalgebra morphism
 $\psi\circ \eta$ from $U(P(V))$ to $V_{\bf 1}$.
Note that $V_{\bf 1}$ is cocommutative and connected with $P(V_{\bf 1})=P(V)$.
Then by Corollary \ref{B(g)=U(g)} $\psi\circ \eta$ is a coalgebra isomorphism and hence
$\psi$ is a coalgebra isomorphism.
Therefore, we have ${\mathcal V}_{P(V)} \simeq \langle P(V)\rangle=V_{\bf 1}$.
Consequently, $V_{\bf 1}$ is a vertex subalgebra of $V$  and hence a vertex subbialgebra. 

Next, we prove the furthermore assertion. Recall from Proposition \ref{va-P(V)} that 
$V$ is naturally a module for the Lie algebra ${\mathcal {L}}_{P(V)}$.
Let $g\in G(V)$. Set
$$K=\{ w\in V\ |\ \Delta(w)\in U({\mathcal {L}}_{P(V)})V_g\otimes U({\mathcal {L}}_{P(V)})V_g\}.$$
As $V_g$ is a subcoalgebra, we have $V_g\subset K$. Let $v\in P(V),\ n\in \Z,\  w\in K$. Then
\begin{align}
\Delta(v_{n}w)=\Delta(v)_n\Delta(w)=(v_{n}\otimes 1+1\otimes v_{n})\Delta(w)
\in U({\mathcal {L}}_{P(V)})V_g\otimes U({\mathcal {L}}_{P(V)})V_g,
\end{align}
which implies $v_nw\in K$. This proves that $K$ is an ${\mathcal {L}}_{P(V)}$-submodule of $V$. Consequently, we have 
 $U({\mathcal {L}}_{P(V)})V_g\subset K$. 
Thus $U({\mathcal {L}}_{P(V)})V_g$ is a subcoalgebra of $V$.

Recall from Lemma \ref{va-coalgebra-morphism} the coalgebra morphism $\phi: U({\mathcal {L}}_{P(V)})\otimes V\rightarrow V$.
As both $U({\mathcal {L}}_{P(V)})$ and $V_g$ are connected coalgebras, by Theorem \ref{Abe-C2.4.21-T2.4.22} again
$U({\mathcal {L}}_{P(V)})V_g$ as the image of $U({\mathcal {L}}_{P(V)})\otimes V_g$ under $\phi$ is connected. 
Then we have  $U({\mathcal {L}}_{P(V)})V_g\subset V_g$,
proving that $V_g$ is a $U({\mathcal {L}}_{P(V)})$-submodule of $V$. Consequently,  $V_g$ is a submodule of $V$ 
for $V_{\bf 1}$ as a vertex algebra as $V_{\bf 1}=\langle P(V)\rangle$.
\end{proof}

\begin{remark}\label{largest-cocomm}
{\em Note that every coalgebra has a largest cocommutative subcoalgebra, which coincides with the sum of 
all cocommutative subcoalgebras. For any vertex bialgebra $V$, it is straightforward to show that 
the largest cocommutative subcoalgebra of $V$ is a vertex subbialgebra.}
\end{remark}

\begin{lemm}\label{V-sharp-result1}
Let $V$ be a vertex bialgebra. Define $V^{\sharp}$ to be the vertex subalgebra of $V$ generated by $P(V)+ \C[G(V)]$.
Then $V^{\sharp}$ is a vertex subbialgebra which is cocommutative with
$G(V^{\sharp})=G(V)$ and $P(V^{\sharp})=P(V)$.
\end{lemm}

\begin{proof} Note that the vertex subalgebra $\langle P(V)\rangle$ generated by $P(V)$ is canonically isomorphic to ${\mathcal{V}}_{P(V)}$,
which is cocommutative. On the other hand, $\C[G(V)]$ is a cocommutative subcoalgebra of $V$. 
So $\langle P(V)\rangle +\C[G(V)]$ is a cocommutative subcoalgebra of $V$.
Just as in Theorem \ref{main-result1} for proving that $\langle P(V)\rangle$ is a vertex subbialgebra, we see that
$V^{\sharp}$ is a cocommutative vertex subbialgebra.
(More generally,  for any cocommutative subcoalgebra $C$ of $V$,
$\langle C\rangle$ is a cocommutative vertex subbialgebra.)
It can be readily seen that $G(V^{\sharp})=G(V)$ and $P(V^{\sharp})=P(V)$.
\end{proof}

Let $V$ be any vertex bialgebra. By Proposition \ref{va-P(V)}, $P(V)$ is a vertex Lie subalgebra of $V$.
From \cite{P}, there exists a uniquely determined vertex algebra morphism $\psi^V: {\mathcal{V}}_{P(V)}\rightarrow V$ with
 $\psi^V|_{P(V)}=1$.
 
As an immediate consequence of Theorem \ref{main-result1} we have: 

\begin{coro}\label{vla-identification}
Let $V$ be any cocommutative connected vertex bialgebra.
Then the vertex algebra morphism $\psi^V$ from ${\mathcal{V}}_{P(V)}$
to $V$ is a vertex bialgebra isomorphism.
\end{coro}

\begin{remark}\label{category-equivalence}
{\em Let $\mathscr{C}_{ccva}$ denote the category of all cocommutative connected vertex bialgebras
and let $\mathscr{C}_{vla}$ denote the category of all vertex Lie algebras. Denote by $P$ the functor from $\mathscr{C}_{ccva}$
to $\mathscr{C}_{vla}$ with $V\mapsto P(V)$ and by ${\mathcal{V}}$ the functor from $\mathscr{C}_{vla}$ to $\mathscr{C}_{ccva}$
with $C\mapsto {\mathcal{V}}_C$.
In view of Proposition \ref{VC-facts}, $P$ is a left adjoint to ${\mathcal{V}}$. This together with Corollary \ref{vla-identification} implies that
functors $P$ and ${\mathcal{V}}$ are category equivalences.}
\end{remark}

We continue to study the structure of cocommutative vertex bialgebras.

\begin{lemm}\label{invertible-case}
Let $V$ be a cocommutative vertex bialgebra. 
Suppose $g\in G(V)$ such that $g_n=0$ for $n\ge 0$, or equivalently $g$ lies in the center of $V$ as a vertex algebra.
Then $g_{-1}$ is a coalgebra morphism from $V$ to $V$ and 
\begin{align}
g_{-1}V_h\subset V_{gh}\quad \text{ for }h\in G(V).
\end{align}
Furthermore, if $g$ is also invertible, then $g_{-1}$ is a coalgebra isomorphism 
which gives an isomorphism from $V_{\bf 1}$ 
to $V_g$ as ${\mathcal{L}}_{P(V)}$-modules and as coalgebras, and
\begin{align}
V_{g}=U({\mathcal{L}}_{P(V)})g
\end{align}
\end{lemm}

\begin{proof} For any $v\in V$, we have
\begin{align*}
&\varepsilon (g_{-1}v)=\varepsilon(g)\varepsilon(v)=\varepsilon(v),\\
& \Delta(g_{-1}v)=\Delta(Y(g,x)v)|_{x=0}=(Y(g,x)\otimes Y(g,x))\Delta(v)|_{x=0}=(g_{-1}\otimes g_{-1})\Delta(v).
\end{align*}
This shows that $g_{-1}$  is a coalgebra morphism from $V$ to $V$. 
It follows that $g_{-1}V_h\subset V_{gh}$ for $h\in G(V)$ as $g_{-1}V_h$ is a connected subcoalgebra containing $gh$.
If $g$ is also invertible, then we obtain $g_{-1}V_h= V_{gh}$ for $h\in G(V)$. In this case, we get
\begin{align*}
V_{g}=g_{-1}V_{\bf 1}=g_{-1}U({\mathcal{L}}_{P(V)}){\bf 1}=U({\mathcal{L}}_{P(V)})g_{-1}{\bf 1}=U({\mathcal{L}}_{P(V)})g,
\end{align*}
as desired.
\end{proof}

Let $V$ be any vertex algebra. For $a\in V$, set
\begin{align}\label{E-minus}
E^{-}(a,x)=\exp\left(\sum_{n\in \Z_{+}}\frac{1}{n}a_{-n}x^n\right).
\end{align}

Using Theorem \ref{main-result1} and Lemma \ref{invertible-case}, we have:

\begin{theo}\label{thm-second}
Let $V$ be a cocommutative vertex bialgebra such that $G(V)$ is a group which lies in the center of $V$ as a vertex algebra.
Then
\begin{align}
V={\mathcal V}_{P(V)}\otimes \C[G(V)]
\end{align}
as a coalgebra, and for $v\in {\mathcal V}_{P(V)},\  \alpha \in G(V)$,
\begin{align}
Y(v\otimes e^{\alpha},x)=E^{-}(\phi(\alpha),x)Y(v,x)\otimes e^{\alpha},
\end{align}
 where $\phi(\alpha)=(\alpha^{-1})_{-1}{\mathcal{D}}(\alpha)\in V$.
Furthermore, $\phi(\alpha)\in P(V)$ for $\alpha\in G(V)$ and 
\begin{align}
\phi(\alpha\beta)=\phi(\alpha)+\phi(\beta)\quad \text{ for }\alpha,\beta\in G(V).
\end{align}
\end{theo}

\begin{proof} From Theorem \ref{main-result1} and Lemma \ref{invertible-case}, we have
$V=\oplus_{g\in G(V)}V_g$, where $V_{\bf 1}={\mathcal V}_{P(V)}$ as vertex bialgebras and $V_g=g_{-1}V_{\bf 1}$ as 
coalgebras and $V_{\bf 1}$-modules for $g\in G(V)$.  
Thus we have $V={\mathcal V}_{P(V)}\otimes \C[G(V)]$ as coalgebras and as ${\mathcal V}_{P(V)}$-modules
with $V_g$ identified with $V_{\bf 1}\otimes g$ for $g\in G(V)$. 

Note that for any central element $u$ of $V$, as $u_n=0$  for $n\ge 0$ we have
$$(u\otimes u)_{-1}={\rm Res}_x x^{-1}(Y(u,x)\otimes Y(u,x))=u_{-1}\otimes u_{-1}.$$
Let $g\in G(V)$. Then
\begin{align*}
&\Delta((g^{-1})_{-1}{\mathcal{D}}(g))=(g^{-1}\otimes g^{-1})_{-1}\Delta ({\mathcal{D}}(g))=((g^{-1})_{-1}\otimes (g^{-1})_{-1})({\mathcal{D}}\otimes 1+1\otimes {\mathcal{D}})(g\otimes g)\\
&=\ (g^{-1})_{-1}{\mathcal{D}}(g)\otimes {\bf 1}+{\bf 1}\otimes (g^{-1})_{-1}{\mathcal{D}}(g)),
\end{align*}
which shows that $(g^{-1})_{-1}{\mathcal{D}}(g)\in P(V)$. Define a map
$$\phi: G(V)\rightarrow P(V);\ \phi(g)=g^{-1}{\mathcal{D}}(g):=(g^{-1})_{-1}{\mathcal{D}}(g).$$
For $g,h\in G(V)$, we have
\begin{align*}
\phi(gh)=(gh)^{-1}{\mathcal{D}}(gh)=g^{-1}h^{-1}({\mathcal{D}}(g)h+g{\mathcal{D}}(h))=g^{-1}{\mathcal{D}}(g)+h^{-1}{\mathcal{D}}(h)
=\phi(g)+\phi(h).
\end{align*}
As $G(V)$ lies in the center of $V$, $A:=\langle G(V)\rangle$ is a commutative vertex subalgebra. 
Then $A$ is a commutative associative algebra
with $ab=a_{-1}b$ for $a,b\in A$ and with a derivation $\partial:={\mathcal{D}}$ such that
$$Y(a,x)b=(e^{x\partial}a)b.$$ 
For $g\in G(V)$, noticing that $\partial g=\phi(g)g$ we have
$$\frac{d}{dx}e^{x\partial}g= e^{x\partial}\partial g=e^{x\partial}(\phi(g)g)=(e^{x\partial}\phi(g))(e^{x\partial}g), $$
where
$$e^{x\partial}\phi(g)=\sum_{n\ge 0}\frac{1}{n!}x^n\partial^n(\phi(g))
=\frac{d}{dx}\sum_{n\ge 0}\frac{1}{n+1}x^{n+1}\frac{1}{n!}\partial^n(\phi(g))=\frac{d}{dx}\sum_{n\ge 1}\frac{1}{n}x^{n}\phi(g)_{-n}.$$
It follows that
$$Y(g,x)=e^{x\partial}g=g\exp \left(\sum_{n\ge 1}\frac{1}{n}x^n\phi(g)_{-n}\right)=E^{-}(\phi(g),x)g.$$
That is,
\begin{align}
Y(g,x)=E^{-}(\phi(g),x)\otimes g.
\end{align}
Furthermore, for any $v\in V$, as $[Y(g,x_1),Y(v,x_2)]=0$ we get
$$Y(v\otimes g,x)=Y(g_{-1}v,x)=Y(g,x)Y(v,x)=E^{-}(\phi(g),x)Y(v,x)\otimes g,$$
as desired.
\end{proof}

On the other hand, we have the following result: 

\begin{theo}\label{mixed-vla-group}
Let $V$ be a vertex algebra, $L$ an abelian semigroup with identity $0$, and 
$\phi:\ L\rightarrow V$ an additive map such that $\phi(L)$ lies in the center of $V$.
 For $v\in V,\  \alpha\in L$, define
\begin{align}
Y(v\otimes e^{\alpha},x)=E^{-}(\phi(\alpha),x)Y(v,x)\otimes e^{\alpha}
\end{align}
on $V\otimes \C[L]$. 
Then $V\otimes \C[L]$ is a vertex algebra with ${\bf 1}\otimes e^{0}$ as its vacuum vector,
which contains $V$ as a vertex subalgebra. Denote this vertex algebra by $V\otimes_{\phi} \C[L]$.
Furthermore, if $V$ is a (cocommutative) vertex bialgebra with $\phi(L)\subset P(V)$, 
then $V\otimes_{\phi} \C[L]$ equipped with the tensor product coalgebra structure
  is a (cocommutative) vertex bialgebra with 
\begin{align}
G(V\otimes_{\phi} \C[L])=G(V)\times L\  \  \text{and }\  P(V\otimes_{\phi} \C[L])=P(V).
\end{align}
\end{theo}

\begin{proof} From definition, we have $Y({\bf 1}\otimes e^{0},x)=1$  as $\phi(0)=0$, and 
$$Y(v\otimes e^{\alpha},x)({\bf 1}\otimes e^{0})=E^{-}(\phi(\alpha),x)Y(v,x){\bf 1}\otimes e^{\alpha}\in (V\otimes \C[L])[[x]],$$
$$\lim_{x\rightarrow 0}Y(v\otimes e^{\alpha},x)({\bf 1}\otimes e^{0})=v\otimes e^{\alpha}.$$
For $u,v\in V,\ \alpha,\beta\in L$, we have
 \begin{align*}
 Y(u\otimes e^{\alpha},x_1)Y(v\otimes e^{\beta},x_2)
  =E^{-}(\phi(\alpha),x_1)E^{-}(\phi(\beta),x_2)Y(u,x_1)Y(v,x_2)\otimes e^{\alpha+\beta}.
 %&=\ E^{-}(\phi(\alpha),x_1)E^{-}(\phi(\beta),x_2)Y(u,x_1)Y(v,x_2)\otimes e^{\alpha+\beta}
 \end{align*}
Then 
\begin{eqnarray*}
&&x_0^{-1}\delta\left(\frac{x_1-x_2}{x_0}\right)Y(u\otimes e^{\alpha},x_1)Y(v\otimes e^{\beta},x_2)\\
&&\quad -x_0^{-1}\delta\left(\frac{x_2-x_1}{-x_0}\right)Y(v\otimes e^{\beta},x_2)Y(u\otimes e^{\alpha},x_1)\\
&=&x_1^{-1}\delta\left(\frac{x_2+x_0}{x_1}\right)E^{-}(\phi(\alpha),x_1)E^{-}(\phi(\beta),x_2) Y(Y(u,x_0)v,x_2)\otimes e^{\alpha+\beta}\\
&=&x_1^{-1}\delta\left(\frac{x_2+x_0}{x_1}\right)E^{-}(\phi(\alpha),x_2+x_0)E^{-}(\phi(\beta),x_2) Y(Y(u,x_0)v,x_2)\otimes e^{\alpha+\beta}.
\end{eqnarray*}
On the other hand, we have
\begin{eqnarray*}
&&x_1^{-1}\delta\left(\frac{x_2+x_0}{x_1}\right)Y\left(Y(u\otimes e^{\alpha},x_0)(v\otimes e^{\beta}),x_2\right)\\
&=&x_1^{-1}\delta\left(\frac{x_2+x_0}{x_1}\right)Y\left(E^{-}(\phi(\alpha),x_0)Y(u,x_0)v\otimes e^{\alpha+\beta},x_2\right)\\
&=&x_1^{-1}\delta\left(\frac{x_2+x_0}{x_1}\right)E^{-}(\phi(\alpha+\beta),x_2)
Y\left(E^{-}(\phi(\alpha),x_0)Y(u,x_0)v,x_2\right)\otimes e^{\alpha+\beta}\\ 
&=&x_1^{-1}\delta\left(\frac{x_2+x_0}{x_1}\right)E^{-}(\phi(\alpha),x_2)E^{-}(\phi(\beta),x_2)
Y\left(E^{-}(\phi(\alpha),x_0)Y(u,x_0)v,x_2\right)\otimes e^{\alpha+\beta}. 
\end{eqnarray*}
We see that the Jacobi identity follows if we can show
$$Y\left(E^{-}(\phi(\alpha),x_0)Y(u,x_0)v,x_2\right)=E^{-}(\phi(\alpha),x_2+x_0)E^{-}(-\phi(\alpha),x_2) Y(Y(u,x_0)v,x_2).$$

Let $a$ be a central element of $V$. Note that ${\mathcal{D}}^ra$  for all $r\ge 0$ are also central. Set 
$$A(x)=\sum_{n\ge 1}\frac{1}{n}a_{-n}x^{n}.$$
As $a_{m}=0$ for $m\ge 0$, we have
$$\frac{d}{dx}A(x)=\sum_{n\ge 1}a_{-n}x^{n-1}=Y(a,x).$$
Furthermore, for any $w\in V$ we have
\begin{eqnarray*}
&&Y(A(x_0)w,x)=\sum_{n\ge 1}\frac{1}{n}x_0^nY(a_{-n}w,x)=\sum_{n\ge 1}\frac{1}{n!}x_0^nY( ({\mathcal{D}}^{n-1}a)_{-1}w,x)\\
&=&\left(\sum_{n\ge 1}\frac{1}{n!}x_0^n\left(\frac{d}{dx}\right)^{n-1}Y(a,x)\right)Y(w,x)\\
&=&\left(\sum_{n\ge 1}\frac{1}{n!}x_0^n\left(\frac{d}{dx}\right)^{n}A(x)\right)Y(w,x)\\
&=&\left(e^{x_0\frac{d}{dx}}A(x)-A(x)\right)Y(w,x)\\
&=&(A(x+x_0)-A(x))Y(w,x).
\end{eqnarray*}
Then we obtain
\begin{align*}
Y(E^{-}(a,x_0)w,x_2)=E^{-}(a,x_2+x_0)E^{-}(-a,x_2)Y(w,x_2)
\end{align*}
as desired. This proves the first assertion. 

For the second assertion, assume that $V$ is a (cocommutative) vertex bialgebra. Then $V\otimes \C[L]$ is 
a (cocommutative) coalgebra. It remains to show that the coalgebra structure maps are vertex algebra morphisms.
Let $u,v\in V,\ \alpha,\beta\in L$. As $\phi(\alpha)\in P(V)$, we have $\varepsilon(\phi(\alpha))=0$, so that
$$\varepsilon(\phi(\alpha)_mw)=\delta_{m,-1}\varepsilon(\phi(\alpha))\varepsilon(w)=0$$
for any $w\in V,\ m\in \Z$.  Then
\begin{align*}
&\varepsilon(Y(u\otimes e^{\alpha},x)(v\otimes e^{\beta}))=\varepsilon\left(E^{-}(\phi(\alpha),x)Y(u,x)v\right) \varepsilon(e^{\alpha+\beta})
= \varepsilon(u)\varepsilon(v)\\
&\quad = \varepsilon(u\otimes e^{\alpha})\varepsilon(v\otimes e^{\beta}),
\end{align*}
noticing that $\varepsilon(e^{\gamma})=1$ for any $\gamma\in L$.
On the other hand, as 
$$\Delta(\phi(\alpha)_mw)=(\phi(\alpha)_m\otimes 1+1\otimes \phi(\alpha)_m)\Delta(w)$$
for any $w\in V,\ m\in \Z$, we have
\begin{align*}
&T^{23}\Delta(Y(u\otimes e^{\alpha},x)(v\otimes e^{\beta}))\\
=\ &\Delta\left(E^{-}(\phi(\alpha),x)Y(u,x)v\right)\otimes \Delta(e^{\alpha+\beta})\\
=\ & (E^{-}(\phi(\alpha),x)\otimes E^{-}(\phi(\alpha),x)) Y(\Delta(u),x)\Delta(v)\otimes  (e^{\alpha+\beta}\otimes e^{\alpha+\beta})\\
=\ & T^{23}Y(\Delta(u\otimes e^{\alpha}),x)\Delta(v\otimes e^{\beta}).
\end{align*}
It follows that $\varepsilon$ and $\Delta$ are vertex algebra morphisms. Therefore, $V\otimes_{\phi} \C[L]$ 
  is a (cocommutative) vertex bialgebra.
  
 For the last assertion, it is clear that $G(V)\times L\subset  G(V\otimes_{\phi} \C[L])$. Suppose $u\in G(V\otimes_{\phi} \C[L])$. 
 Write $u=\sum_{\alpha\in S}u_{\alpha}\otimes e^\alpha$, where $S$ is a finite subset of $L$ and 
 $u_{\alpha}$ for $\alpha\in S$ are nonzero vectors in $V$. We have
 \begin{align*}
 &T^{23}\Delta(u)=T^{23}(u\otimes u)=\sum_{\alpha,\beta\in S}u_\alpha \otimes u_\beta\otimes e^\alpha\otimes e^\beta,\\
 &T^{23}\Delta(u)=\sum_{\alpha\in S}\Delta(u_{\alpha})\otimes e^\alpha\otimes e^\alpha,
 \end{align*}
 which imply $u_{\alpha}\otimes u_{\beta}=0$ in $V\otimes V$ for $\alpha\ne \beta$. 
 Then $S$ must be a one-element set. Therefore, $u=u_\alpha\otimes e^\alpha$ for some $\alpha\in L$ and 
 $\Delta(u_\alpha)=u_\alpha\otimes u_\alpha$. This shows $u\in G(V)\times L$. 
 Therefore, we have $G(V\otimes_{\phi} \C[L])=G(V)\times L$. Similarly, we can show $P(V\otimes_{\phi} \C[L])=P(V)$.
\end{proof}

The following is a property of the vertex bialgebra $V\otimes_{\phi} \C[L]$: 

\begin{prop}\label{universal-V-phi-L}
Assume that $V, L, \phi$ are given as in Theorem \ref{mixed-vla-group} (with all the assumptions).
Let $U$ be any vertex bialgebra with a semigroup morphism $\psi_L:  L\rightarrow G(U)$ and 
a vertex bialgebra  morphism $\psi_V: V\rightarrow U$ such that 
\begin{align}
&[\psi_L(\alpha)_m,\psi_V(v)_n]=0\quad \text{  for }\alpha\in L,\ v\in V,\ m,n\in \Z,\\
&{\mathcal{D}}(\psi_L(\alpha))= \psi_L(\alpha)_{-1}\psi_V(\phi(\alpha)) \quad \text{ for }\alpha\in L.
\end{align}
Define a linear map $\Psi:  V\otimes \C[L]\rightarrow U$ by
$$\Psi(v\otimes e^{\alpha})=\psi_L(\alpha)_{-1}\psi_V(v)=\psi_V(v)_{-1}\psi_L(\alpha)
\quad \text{ for }v\in V,\ \alpha\in L.$$
Then $\Psi$ is a vertex bialgebra morphism from $V\otimes_{\phi} \C[L]$ to $U$.
% with $\Psi|_V=\psi_V$ and $\Psi(e^{\alpha})=\psi_L(\alpha)$ for $\alpha\in L$.
\end{prop}

\begin{proof}  For  $v\in V,\ \alpha\in L$, we have 
\begin{align*}
\varepsilon \Psi(v\otimes e^{\alpha})=\varepsilon(\psi_L(\alpha)_{-1}\psi_V(v))
=\varepsilon(\psi_L(\alpha))\varepsilon(\psi_V(v))=\varepsilon(v)=\varepsilon(v)\varepsilon (e^{\alpha})
=\varepsilon (v\otimes e^{\alpha})
\end{align*}
and
\begin{align*}
&\Delta \Psi(v\otimes e^{\alpha})=  \Delta(\psi_L(\alpha)_{-1}\psi_V(v))
= (\Delta \psi_L(\alpha))_{-1}\Delta\psi_V(v)\\
&=  (\psi_L(\alpha)_{-1}\otimes \psi_L(\alpha)_{-1})(\psi_V\otimes \psi_V)\Delta(v)
= (\Psi\otimes \Psi)\Delta(v\otimes e^{\alpha}).
\end{align*}
Then $\Psi$ is a coalgebra morphism. Next, we show that $\Psi$ is also a vertex algebra morphism. 
Note that
$$\Psi ({\bf 1}\otimes e^{0})=(\psi_L(0))_{-1}\psi_V({\bf 1})={\bf 1}_{-1}{\bf 1}={\bf 1}\in U.$$
Let $u,v\in V,\ \alpha,\beta\in L$. We have
\begin{align*}
&\Psi\left(Y(u\otimes e^{\alpha})(v\otimes e^{\beta})\right)\\
=& \Psi\left(E^{-}(\phi(\alpha),x)Y(u,x)v\otimes e^{\alpha+\beta}\right) \\
=& (\psi_L(\alpha+\beta))_{-1}\psi_V\left(E^{-}(\phi(\alpha),x)Y(u,x)v\right)\\
=&(\psi_L(\alpha+\beta))_{-1}E^{-}(\psi_V(\phi(\alpha)),x)\psi_V\left(Y(u,x)v\right)
\end{align*}
and
\begin{align*}
Y(\Psi(u\otimes e^{\alpha}),x)\Psi(v\otimes e^{\beta})
=Y\left(\psi_L(\alpha)_{-1}\psi_V(v),x\right)\psi_L(\beta)_{-1}\psi_V(v).
\end{align*}
Since $[\psi_L(\alpha)_{m},\psi_V(u)_n]=0$ for all $m,n\in \Z$, we get
$$Y\left(\psi_L(\alpha)_{-1}\psi_V(u),x\right)=Y(\psi_L(\alpha),x)Y(\psi_V(u),x).$$
With ${\mathcal{D}}(\psi_L(\alpha))= \psi_L(\alpha)_{-1}\psi_V(\phi(\alpha))$ we have
\begin{align*}
&Y(\psi_L(\alpha),x)=E^{-}(\psi_V(\phi(\alpha)),x)\psi_L(\alpha)_{-1},\\
&\psi_L(\alpha)_{-1}\psi_L(\beta)_{-1}=\psi_L(\alpha+\beta)_{-1}
\end{align*}
on $\psi_V(V)$. Then
\begin{align*}
Y(\Psi(u\otimes e^{\alpha}),x)\Psi(v\otimes e^{\beta})& =Y(\psi_L(\alpha),x)Y(\psi_V(u),x)\psi_L(\beta)_{-1}\psi_V(v)\\
&=Y(\psi_L(\alpha),x)\psi_L(\beta)_{-1}Y(\psi_V(u),x)\psi_V(v)\\
&=E^{-}(\psi_V(\phi(\alpha)),x)\psi_L(\alpha)_{-1}\psi_L(\beta)_{-1}\psi_V(Y(u,x)v)\\
&=E^{-}(\psi_V(\phi(\alpha)),x)\psi_L(\alpha+\beta)_{-1}\psi_V(Y(u,x)v).
\end{align*}
It follows that $\Psi$ is indeed a vertex algebra morphism. 
Therefore,  $\Psi$ is a vertex bialgebra morphism.
\end{proof}

For any subsets $A$ and $B$ of a vertex algebra $V$, set
\begin{align}
C_{A}(B)=\{ b\in B\ |\  [Y(b,x_1),Y(a,x_2)]=0\  \text{ for all }a\in A\}.
\end{align}
As an immediate consequence of Proposition \ref{universal-V-phi-L}, we have: 

\begin{coro}\label{morphism-image}
 Let $V$ be a vertex bialgebra satisfying the conditions that $G(V)\subset C_V(V)$ and that
 there exists a map $\phi: G(V)\rightarrow C_{P(V)}(P(V))$ such that 
 \begin{align}
 & \phi(gh)=\phi(g)+\phi(h)\quad \text{  for }g,h\in G(V),\\
&{\mathcal{D}}(g)=g_{-1}\phi(g)\quad \text{  for } g\in G(V).
 \end{align}
 Then there exists a vertex bialgebra morphism $\Psi: {\mathcal{V}}_{P(V)}\otimes_{\phi} \C[G(V)]\rightarrow V$
 with $\Psi|_{P(V)+\C[G(V)]}=1$.
 Furthermore,  if $V$ as a vertex algebra is generated by $P(V)+\C [G(V)]$, $\Psi$ is surjective.
 \end{coro}

As a special case of Theorem \ref{mixed-vla-group} we have: 

\begin{prop}\label{L-C-vba}
Let $C$ be a vertex Lie algebra, $L$ an abelian semigroup with identity $0$, 
and let $\phi: L\rightarrow C$ be an additive map such that
\begin{align}
\phi(\alpha)_{n}b=0
\quad \text{ for }\alpha\in L,\ b\in C,\ n\in \N.
\end{align}
Set
\begin{align}
{\mathcal{V}}(C,L,\phi)={\mathcal{V}}_C\otimes \C[L],
\end{align}
equipped with the tensor product coalgebra structure and the vertex algebra structure obtained in Proposition \ref{mixed-vla-group}.
Then  ${\mathcal{V}}(C,L,\phi)$ is a cocommutative vertex bialgebra with 
\begin{align}
G({\mathcal{V}}(C,L,\phi))=L\  \  \text{and }\  P({\mathcal{V}}(C,L,\phi))=C.
\end{align}
\end{prop}

\begin{proof} Let $\alpha\in L$. As $\phi(\alpha)_ib=0$ for $b\in C,\ i\in \N$, we have 
$[\phi(\alpha)(m),b(n)]=0$ in ${\mathcal{L}}_C$ for all $m,n\in \Z$. 
It follows that $\phi(\alpha)$ lies in the center of the associated vertex algebra ${\mathcal{V}}_C$. Then 
by Proposition \ref{mixed-vla-group},  ${\mathcal{V}}(C,L,\phi)$ is a cocommutative vertex bialgebra.
The rest is clear. 
\end{proof}

As another special case of Theorem \ref{mixed-vla-group}, we immediately have:

\begin{coro}\label{A-CL}
Let $L$ be an abelian semigroup with identity $0$ and let $A$ be any commutative differential algebra 
with an additive map $\phi: L\rightarrow A$.
Then $A\otimes \C[L]$ is a vertex algebra with
\begin{align}
Y(a\otimes e^{\alpha})=E^{-}(\phi(\alpha),x)e^{x\partial}a\otimes e^{\alpha}\quad \text{ for }a\in A,\ \alpha\in L,
\end{align}
where
$$E^{-}(\phi(\alpha),x)=\exp\left(\sum_{n\ge1}\frac{1}{n!}(\partial^{n-1} \phi(a))x^n\right). $$
\end{coro}

As an immediate consequence of Theorem \ref{thm-second}, we have: 

\begin{coro}\label{B-factorization}
Let $B$ be a cocommutative and commutative differential bialgebra such that $G(B)$ is a group. Then  
\begin{align} 
B\simeq S(P(B))\otimes_{\phi} \C[G(B)],
\end{align}
where $\phi: G(B)\rightarrow P(B)$ with $\phi(g)=g^{-1}\partial (g)$ for $g\in G(B)$.
\end{coro}

\section{Differential bialgebras associated to abelian semigroups}

In this section, we study cocommutative and commutative differential bialgebras, 
which determine all commutative and cocommutative vertex bialgebras through Borcherds' construction.
More specifically, we associate a cocommutative and commutative differential bialgebra to any abelian semigroup $L$ with identity 
and we give a universal property and a characterization of $B_L$.

We begin by recalling the differential bialgebra associated to an abelian group, 
which is rooted in the construction of lattice vertex algebras.
Let $L$ be any (additive) abelian group. The group algebra $\C[L]$ is known to be a Hopf algebra.
On the other hand, set $\mathfrak{h}=\C\otimes_{\Z}L$, a vector space, and set
\begin{align}
\bar{\alpha}=1\otimes \alpha\in \mathfrak{h}\quad  \text{for }\alpha\in L.
\end{align}
Furthermore, set $\widehat{\mathfrak{h}}^{-}=\mathfrak{h}\otimes t^{-1}\C[t^{-1}]$.
View $\widehat{\mathfrak{h}}^{-}$ as an abelian Lie algebra, so the universal enveloping algebra $U(\widehat{\mathfrak{h}}^{-})$
coincides with the symmetric algebra $S(\widehat{\mathfrak{h}}^{-})$.
Set
\begin{align}
B_{L}=S(\widehat{\mathfrak{h}}^{-})\otimes \C[L],
\end{align}
which is naturally a Hopf algebra, in particular, a bialgebra.
As an associative algebra, $B_{L}$ admits a derivation $\partial$ which is uniquely determined  by
\begin{align}\label{pa-de-}
\partial e^{\alpha}=\bar{\alpha}(-1)\otimes e^{\alpha},\quad \partial a(-n)=na(-n-1)
\quad \text{ for }\alpha\in L,\ a\in \mathfrak{h},\ n\in \Z_{+},
\end{align}
where $a(-n)=a\otimes t^{-n}$. This makes $B_L$ a differential algebra.
Furthermore, $B_L$ is a cocommutative (and commutative) differential bialgebra  (see \cite{L}).

\begin{remark}\label{BL-property}
{\em Note that as a cocommutative coalgebra over $\mathbb C$, $B_L$ is pointed.
It is straightforward (by using $\Delta$ and $\varepsilon$)  to show that $G(B_L)=L$ and $P(B_L)=\widehat{\mathfrak{h}}^{-}$.
 From (\ref{pa-de-}), it can be readily seen that $B_L$  as a differential algebra is generated by the subbialgebra $\C[L]$.}
\end{remark}

%Note that $B_L$ is non-degenerate in the sense of Etingof and Kazhdan if and only if $L$ is torsion free.

The following is a universal property of the differential bialgebra $B_{L}$, extending the 
universal property of $B_{L}$ as a differential algebra in \cite{L}: 

\begin{prop}\label{BL bialgebra universal property}
Assume that $L$ is an abelian group.
Let $A$ be any commutative differential bialgebra and let $\psi: \mathbb{C}[L]\rightarrow A$ be a bialgebra morphism. 
Then there exists a unique differential bialgebra morphism  $\widetilde{\psi}: B_L\rightarrow A$ with
$\widetilde{\psi}|_{\mathbb{C}[L]}=\psi$.
\end{prop}

\begin{proof} From  \cite{L}, there exists a unique differential algebra morphism $\widetilde{\psi}: B_L\rightarrow A$ with
$\widetilde{\psi}|_{\mathbb{C}[L]}=\psi$.
It remains to show the  $\widetilde{\psi}$ is also a coalgebra morphism.
Set $$K=\{x\in B_L\mid \Delta_{A}\widetilde{\psi}(x)=(\widetilde{\psi}\otimes \widetilde{\psi})\Delta_{B_{L}}(x)\}.$$	
As $\widetilde{\psi}$, $\Delta_A$ and $\Delta_{B_L}$ are differential algebra morphisms,
 it is straightforward to show that $K$ is a differential subalgebra of $B_L$.
Since $\psi$ is a coalgebra morphism, we have $\C[L]\subset K$. 
It follows that $K=B_L$. This proves $\Delta_{A} \widetilde{\psi}=(\widetilde{\psi}\otimes \widetilde{\psi})\Delta_{B_{L}}$. 
Similarly, we have $\varepsilon_{A} \widetilde{\psi}=\varepsilon_{B_{L}}$. Therefore,  
$\widetilde{\psi}$  is a coalgebra morphism and hence a differential bialgebra morphism. 
\end{proof}

Let $B$ be a commutative differential bialgebra with derivation $\partial$. For $v\in P(B)$, we have 
\begin{align*}
\Delta(\partial v)=(\partial\otimes 1+1\otimes \partial)\Delta(v)
=(\partial\otimes 1+1\otimes \partial)(v\otimes 1+1\otimes v)=\partial v\otimes 1+1\otimes \partial v,
\end{align*}
which implies $\partial v\in P(B)$. Thus $\partial P(B)\subset P(B)$, so that
$P(B)$ is a $\C[\partial]$-module. Here, we consider $\C[\partial]$ as the polynomial algebra with
$\partial$ as an indeterminate. As usual, a $\C[\partial]$-module $W$ is said to be {\em torsion-free} if
$f(\partial)w\ne 0$ for any nonzero $f(\partial)\in \C[\partial],\ w\in W$.
On the other hand, assume that $G(B)$ is a group.
From Theorem \ref{thm-second}, we have a linear map 
$\phi: \C\otimes_{\Z}G(B)\rightarrow B$ with $\phi(g)=g^{-1}\partial (g)$ for $g\in G(B)$.

 \begin{prop}\label{classification-BL}
Let $B$ be a commutative differential  bialgebra such that $G(B)$ is a group which 
generates $B$ as a differential algebra. Suppose that the linear map
$\phi: \C\otimes_{\Z}G(B)\rightarrow B$ with $\phi(g)=g^{-1}\partial (g)$ for $g\in G(B)$ is injective 
and $P(B)$ is a torsion-free $\C[\partial]$-module.
Then  $B\simeq B_{G(B)}$ as a differential bialgebra.
 \end{prop}

 \begin{proof} Let $L$ be the additive copy of $G(B)$.   
 By Proposition \ref{BL bialgebra universal property}, there exists a differential bialgebra
 morphism $\psi: B_{L}=S(\widehat{\mathfrak{h}}^-)\otimes \mathbb C[L]\rightarrow B$ 
 such that $\psi(e^\alpha)=\alpha$ for $\alpha\in L$. 
 Since $B$  as a differential algebra is generated by $G(B)$,
 $\psi$ is surjective. Then it remains to prove that $\psi$ is injective.
 
 Recall that $B_{L}$ is a commutative and cocommutative (pointed) bialgebra with
$G(B_L)=L$ and $P(B_L)=\widehat{\mathfrak{h}}^{-}$. 
On the other hand, noticing that $B$ as a homomorphism image of $B_L$  is cocommutative,  
by Corollary \ref{B-factorization}, we have $B=S(P(B)) \otimes \mathbb{C}[L]$ as bialgebras. 
As a bialgebra morphism, $\psi$ is also a $\C[\partial]$-module morphism, 
sending $P(B_L)=\widehat{\mathfrak{h}}^{-}$ to $P(B)$.

Notice that  $\widehat{\mathfrak{h}}^{-}$ is a free $\C[\partial]$-module on ${\mathfrak{h}}={\mathfrak{h}}(-1)$.
For $\alpha\in L=G(B)$, we have
 $$\psi(\bar{\alpha}(-1))=\psi(e^{-\alpha}\partial e^{\alpha})=\alpha^{-1}\partial \alpha\in P(B),$$
 recalling that $\bar{\alpha}=1\otimes \alpha\in {\mathfrak{h}}$.
 From our assumption on $\phi$, $\psi$ is an injective map from ${\mathfrak{h}}(-1)$ to $P(B)$.
Furthermore, as $P(B)$ is assumed to be a torsion-free $\C[\partial]$-module, 
 $\psi$ is injective from $\widehat{\mathfrak{h}}^-=P(B_L)$ to $P(B)=P(S(P(B)))$.
Then by Corollary \ref{B(g)=U(g)}  $\psi: S(\widehat{\mathfrak{h}}^-) \rightarrow S(P(B))$ is injective. 
Consequently, $\psi$ is an injective map from $B_L$ to $B$.
Therefore, $\psi$ is an isomorphism  of differential bialgebras.
\end{proof}

Next, we generalize this construction to abelian semigroups.
 Assume now that $L$ is an additive (abelian) semigroup with identity $0$.  
 Let $\C L$ denote the vector space over $\C$ with a basis $\{ (1,\alpha)\ |\ \alpha\in L\}$.
 Define $\C\otimes_{\N}L$ to be the quotient space of $\C L$
 modulo the subspace spanned by vectors
 $$(1,\alpha)+(1,\beta)-(1,\alpha+\beta)$$
 for $\alpha,\beta\in L$.
 Set  
 \begin{align}
 \mathfrak{h}=\C\otimes_{\N}L
 \end{align}
and write $\bar{\alpha}$ for the image of $(1,\alpha)$ in $ \mathfrak{h}$.
 Notice that in case $L$ is a group, we have $\C\otimes_{\N}L=\C\otimes_{\Z}L$.
 
 Just as before,  we get a commutative and cocommutative differential bialgebra $B_L$.
  Identify $a\in \mathfrak{h}$ with $a(-1)\otimes 1\in B_L$, 
 to view $\mathfrak{h}$ as a subspace of $B_{L}$. It is clear that $B_L$  as a differential algebra 
 is  generated by $\C[L]+\mathfrak{h}$.  (In general,  $B_L$  as a differential algebra may not be generated by $\C[L]$ alone.)

Define a map 
\begin{align}
\phi:\  L\rightarrow \widehat{\mathfrak{h}}^{-}=P(S(\widehat{\mathfrak{h}}^{-}));\quad  
\alpha\mapsto \phi(\alpha)=\bar{\alpha}(-1)\quad \text{ for }\alpha\in L.
\end{align}
We have the following (universal) property for $B_{L}$: 

\begin{prop}\label{semigroup-BL}
Let $B$ be any commutative differential bialgebra, and let $\psi:\  \C[L]\rightarrow B$ be a bialgebra morphism, 
$\phi_B:\  \mathfrak{h}\rightarrow P(B)$ a linear map such that
\begin{align}
\partial \psi(e^{\alpha})=\phi_B(\bar{\alpha})\psi(e^{\alpha})\quad \text{ for }\alpha\in L.
\end{align}
Then there exists a unique differential bialgebra morphism $f: B_L\rightarrow B$ such that
$f|_{\C[L]}=\psi$ and $f|_{\mathfrak{h}}=\phi_B$.
\end{prop}

\begin{proof} We only need to prove the existence. 
Using Corollary \ref{A-CL}, we can easily see that $B_L=S(\widehat{\mathfrak{h}}^{-})\otimes_{\phi}\C[L]$ as differential bialgebras. 
Since $\widehat{\mathfrak{h}}^{-}$ is a free $\C[\partial]$-module over $\mathfrak{h}$,
 the linear map $\phi_B: \mathfrak{h}\rightarrow P(B)$ extends uniquely to a $\C[\partial]$-module morphism 
from $\widehat{\mathfrak{h}}^{-}$ to $P(B)$ and then to an algebra morphism
$\tilde{\phi}_B: \ S(\widehat{\mathfrak{h}}^{-})\rightarrow B$. 
It is straightforward to show that $\tilde{\phi}_B$ is also a coalgebra morphism.
Then by Proposition \ref{universal-V-phi-L}
we have a differential bialgebra morphism $f:\ B_L\rightarrow B$ such that 
$$f(X\otimes e^{\alpha})=\tilde{\phi}_B(X)\psi(e^{\alpha})\quad \text{ for }X\in S(\widehat{\mathfrak{h}}^{-}),\ \alpha\in L,$$
which implies $f|_{\C[L]}=\psi$ and $f|_{\mathfrak{h}}=\phi_B$. 
%For $\alpha\in L$, we have
%$$f(\partial e^{\alpha})=f(\bar{\alpha}(-1)e^{\alpha})=f(\bar{\alpha}(-1))f(e^{\alpha})
 %=\phi_B(\bar{\alpha})\psi(e^{\alpha})=\partial \psi(e^{\alpha})=\partial f(e^{\alpha}).$$
\end{proof}

Using essentially the same arguments as in the proof of Proposition \ref{classification-BL},
  we obtain the following characterization of the differential bialgebra $B_{L}$:

\begin{prop}\label{identification-BL}
Let $B$ be a commutative differential bialgebra such that $P(B)$ is a torsion-free $\C[\partial]$-module and such that
$P(B)+\C[G(B)]$ generates $B$ as a differential algebra.
In addition, assume that there is an injective linear map $\phi: \C\otimes_{\N}G(B)_{+}\rightarrow P(B)$ such that
\begin{align}
\partial g=\phi(\bar{g})g\quad \text{ for }g\in G(B),
\end{align}
where $G(B)_{+}$ denotes an additive version of $G(B)$ and
$\bar{g}=1\otimes g\in \C\otimes_{\N}G(B)_{+}$.
Then the differential algebra morphism $f: B_{G(B)_{+}}\rightarrow B$ uniquely determined by
$f|_{G(B)}=1$ and $f|_{\mathfrak{h}}=1$ is a differential bialgebra isomorphism.
\end{prop}

\section*{Acknowledgment}
J. Han  is supported by the CSC (grant No. 202006265002).  Y. Xiao is  supported by NSF of China (grant No.  11971350) and  the CSC (grant No. 202006260122).

%\small 

\end{document}